\documentclass[12pt]{article}

\usepackage{amsmath,amsthm,amssymb}
\usepackage[colorlinks,
            linkcolor=red,
            anchorcolor=blue,
            citecolor=magenta
            ]{hyperref}
\usepackage{pdfsync}
\usepackage[numbers,sort&compress]{natbib}
\allowdisplaybreaks

\theoremstyle{definition}

\theoremstyle{plain}
\newtheorem{theorem}{Theorem}[section]

\newtheorem{lemma}{Lemma}[section]

\numberwithin{equation}{section}

\newcommand{\vs}{\vspace}

\begin{document}

\title{Multi-bump solutions for the nonlinear magnetic
Schr\"odinger equation with logarithmic nonlinearity\footnote{  This work was partially supported by NNSFC (No. 11671407 and No. 11801076), FDCT (No. 0091/2018/A3), Guangdong Special Support Program (No. 8-2015) and the key project of NSF of Guangdong Province (No. 2016A030311004).}}

\author{ Jun Wang$^{a}$, Zhaoyang Yin$^{a, b}$\footnote {Corresponding author. wangj937@mail2.sysu.edu.cn (J. Wang), mcsyzy@mail.sysu.edu.cn (Z. Yin)
} \\
{\small $^{a}$Department of Mathematics, Sun Yat-sen University, Guangzhou, 510275, China } \\
{\small $^{b}$School of Science, Shenzhen Campus of Sun Yat-sen University, Shenzhen, 518107, China } \\
}

	\date{}

	\maketitle

\date{}

 \maketitle \vs{-.7cm}

  \begin{abstract}
In this paper, we
study the  following nonlinear magnetic
Schr\"odinger equation with logarithmic nonlinearity
\begin{equation*}
-(\nabla+iA(x))^2u+\lambda V(x)u =|u|^{q-2}u+u\log |u|^2,\ u\in H^1(\mathbb{R}^N,\mathbb{C}),
\end{equation*}
where the magnetic potential $A \in L_{l o c}^2\left(\mathbb{R}^N, \mathbb{R}^N\right)$, $2<q<2^*,\  \lambda>0$ is a parameter and the nonnegative continuous function $V: \mathbb{R}^N \rightarrow \mathbb{R}$ has the deepening potential well. Using the variational methods, we obtain that the equation has at least $2^k-1$ multi-bump solutions when $\lambda>0$ is large enough.
\end{abstract}

{\footnotesize {\bf   Keywords:} Logarithmic Sch\"odinger equations; Multiple solutions;  Variational methods.

}

\section{ Introduction and main results}

This paper studies the existence of multi-bump solutions for the  following nonlinear magnetic
Schr\"odinger equation with logarithmic nonlinearity
\begin{equation} \label{eq1.1}
-(\nabla+iA(x))^2u+\lambda V(x)u =|u|^{q-2}u+u\log |u|^2,\ u\in H^1(\mathbb{R}^N,\mathbb{C}),
\end{equation}
where the magnetic potential $A \in L_{l o c}^2(\mathbb{R}^N, \mathbb{R}^N),\ 2<q< 2^*=\frac{2N}{N-2}$, $\lambda>0$ is a parameter, $i$ is the imaginary unit and $V: \mathbb{R}^N \rightarrow \mathbb{R}$ is the nonnegative continuous function.

Logarithmic Schr\"odinger equation has wide applications in quantum mechanics, quantum optics, nuclear physics, transport and diffusion phenomena, open quantum system, Bose-Einstein condensations, see \cite{{AVKG2011},{IBBJM1976},{EFH1985},{GL2008},{KY1978},{KGZ2010}}. It takes the following form
\begin{equation*}
 \left\{\aligned
& \partial_t v(t, x)=i \Delta v(t, x)+i \lambda v(t, x) \log \left(|v(t, x)|^2\right)+iW\left(t, x,|v|^2\right) v(t, x),\ x \in \mathbb{R}^d,\ t>0, \\
& v(0, x) =v_0(x),\ x\in \mathbb{R}^d,
\endaligned
\right.
\end{equation*}
where $\Delta$ is the Laplacian operator on $\mathcal{M} \subset \mathbb{R}^d$, $\mathcal{M}$ is either a bounded domain with homogeneous Dirichlet boundary condition or $\mathbb{R}^d$, $t$ is time, $x$ is spatial coordinate, $\lambda \in \mathbb{R} \setminus\{0\}$ characterizes the force of nonlinear interaction, and $W$ is a real-valued function. It is worth mentioning that, the logarithmic Schr\"odinger equation was introduced by Mycielski and Bialynicki-Birula \cite{IBBJM1976}.  They proved that it is the only nonlinear theory that holds the separability of non interacting systems: for non interacting subsystems, the nonlinear term does not introduce correlation. In present paper, our goal is to find the existence of standing wave solutions for the  logarithmic Schr\"odinger equation.

The idea of this paper is derived from the following  nonlinear Schr\"odinger equations with deepening potential well
\begin{equation}\label{eq1.2}
 \left\{\aligned
&-\Delta v+(\lambda V(y)+P(y)) v=v^p, & y\in\mathbb{R}^N, \\
& v(y)>0, & y\in\mathbb{R}^N.
\endaligned
\right.
\end{equation}
One of the key lies in the assumption that the first eigenvalue of $-\Delta+P(x)$ on $\Omega_j$ under the Dirichlet boundary condition is positive, where $j \in\{1,2, \ldots, k\},\ p \in\left(1, \frac{N+2}{N-2}\right)$ and $N \geq 3$. In \cite{YHDKT2003}, Ding and Tanaka showed that the problem \eqref{eq1.2} has at least $2^k-1$ multi-bump solutions for $\lambda>0$ large enough. Recently,  Tanaka and Zhang in \cite{KTZC2017} considered the following Schr\"odinger equation with logarithmic nonlinearity
$$
\left\{\begin{array}{l}
-\Delta v+W(x) v=P(y) v \log v^2 \quad \text { in } \mathbb{R}^N, \\
v \in H^1(\mathbb{R}^N),
\end{array}\right.
$$
where $N \geq1$ and $W, P \in C^1(\mathbb{R}^N, \mathbb{R})$ satisfy
\begin{itemize}
\item[$(A_1)$] $W(y), P(y)>0$ for all $y \in \mathbb{R}^N$;

\item[$(A_2)$] $W(y), P(y)$ are $1$-periodic in each $y_i(i=1,2, \cdots, N)$, that is,
$$
\begin{aligned}
& W\left(y_1, \cdots, y_i+1, \cdots, y_N\right)=W\left(y_1, \cdots, y_i, \cdots, y_N\right), \\
& P\left(y_1, \cdots, y_i+1, \cdots, y_N\right)=P\left(y_1, \cdots, y_i, \cdots, y_N\right)
\end{aligned}
$$
for all $y=\left(y_1, y_2, \cdots, y_N\right) \in \mathbb{R}^N$ and $i=1, \cdots, N$.
\end{itemize}
They adopted a method of using the periodic problem of $2L$ and obtained the existence of infinitely many multi-bump solutions,  which are different under the action of $\mathbb{Z}^N$. After that, Alves and Ji in \cite{COACJ2022} studied the existence of multi-bump positive solutions for the following Schr\"odinger equation with logarithmic nonlinearity
\begin{equation*}
\left\{\begin{array}{l}
-\Delta u+\lambda V(x) u=u \log u^2 \quad \text { in } \mathbb{R}^N \\
u \in H^1(\mathbb{R}^N).
\end{array}\right.
\end{equation*}
Note that, they did not assume that condition $(a2)$ in \cite{YHDKT2003} holds. Hence, they needed to modify the nonlinearity in a special way to prove the Palais-Smale condition. It is then quite natural to ask, does \eqref{eq1.1} have the type of multi-bump solutions? Note that, to make the research more interesting, we also study magnetic potential. To our knowledge, this paper is the first to study the existence of multi-bump solutions to the logarithmic Schr\"odinger equation with magnetic potential by using the variational method.

Inspired by \cite{{COACJ2022},{YHDKT2003},{COAIS2023}}, the main purpose of the present paper is to investigate the existence of multi-bump positive solutions. The key to this article is to analyze the interaction between power terms and logarithmic terms. The functional associated with the \eqref{eq3.1} lost some other good properties, so we need to develop a new method to prove the boundedness of the (PS) sequence. In fact, due to the particularity of the equation in \cite{COACJ2022}, they developed a relatively simple method to obtain boundedness of the (PS) sequence, but this method is not applicable to this paper. Moreover, note that the nonlinear term $t\log|t|^2+|u|^q\neq0$ as $t\rightarrow0$, so we cannot apply directly del Pino and Felmer's method in \cite{MDP1996}. In order to get the compactness, we need to modify the penalization in \cite{COACJ2022}.

Our assumptions on $V$ are:
\begin{itemize}
\item[$(V_1)$] $V \in C(\mathbb{R}^N, \mathbb{R})$ and $V(x) \geq 0$.

\item[$(V_2)$] $\Omega:=\operatorname{int} V^{-1}(0)$ is a non-empty bounded open subset with smooth boundary and $ \overline{\Omega}=V^{-1}(0)$, where int $V^{-1}(0)$ denotes the set of the interior points of $V^{-1}(0)$.

\item[$(V_3)$] $\Omega$ consists of $k$ components:
$$
\Omega=\Omega_1 \cup \Omega_2 \cup \cdots \cup \Omega_k
$$
and $\overline{\Omega}_i \cap \overline{\Omega}_j=\emptyset$ for all $i \neq j$.
\end{itemize}

The main results of this paper are as follows.
\begin{theorem}\label{t1.1}
Assume that $N\geq1$ and $(V_1)-(V_3)$ hold. Then for any non-empty subset $\Gamma$ of $\{1,2, \ldots, k\}$, there exists $\lambda^*>0$ such that for all $\lambda \geqslant \lambda^*$, the problem \eqref{eq1.1} has a  nontrivial solution $u_\lambda$. Moreover, the family $\left\{u_\lambda\right\}_{\lambda \geqslant \lambda^*}$ has the following properties: for any sequence $\lambda_n \rightarrow \infty$, we can extract a subsequence $\lambda_{n_i}$ such that $u_{\lambda_{n_i}}$ converges strongly in $H_A^1(\mathbb{R}^N,\mathbb{C})$ to a function $u$ which satisfies $u(x)=0$ for $x \notin \Omega_{\Gamma}$ and the restriction $\left.u\right|_{\Omega_j}$ is a least energy solution of
$$
\begin{cases}-(\nabla+iA(x))^2u=|u|^{q-2}u+u \log |u|^2 & \text { in } \Omega_{\Gamma} \\ u=0 & \text { on } \partial \Omega_{\Gamma},\end{cases}
$$
where $\Omega_{\Gamma}=\bigcup\limits_{j \in \Gamma} \Omega_j$.
\end{theorem}
\begin{theorem}\label{t1.2}
Assume that $N\geq1$ and $(V_1)-(V_3)$ hold, there exists $\lambda_*>0$ such that for all $\lambda \geqslant \lambda_*$, the problem \eqref{eq1.1}  has at least $2^k-1$ solutions.
\end{theorem}

Theorem \ref{t1.2} can be directly obtained from Theorem \ref{t1.1}. The structure of this paper is arranged as follows. In section 2, we provide some notations in the proof of main theorems. In the following three sections, we aim to prove some important lemma. Finally, we get the Theorem \ref{t1.1}.

\section{Preliminary}

In this section, we outline the variational framework for problem \eqref{eq1.1} and give some preliminary lemmas.
For any $v: \mathbb{R}^N \rightarrow \mathbb{C}$, define
$$
\nabla_A v:=(\nabla+i A) v
$$
and
$$
H_A^1(\mathbb{R}^N, \mathbb{C}):=\{u \in L^2(\mathbb{R}^N, \mathbb{C}):|\nabla_A u| \in L^2(\mathbb{R}^N, \mathbb{R})\}
$$
The space $H_A^1(\mathbb{R}^N, \mathbb{C})$ is a Hilbert space endowed with the scalar product
$$
\langle u, v\rangle:=\operatorname{Re} \int_{\mathbb{R}^N}\left(\nabla_A u \overline{\nabla_A v}+u \overline{v}\right) d x  \quad \text { for any } u, v \in H_A^1\left(\mathbb{R}^N, \mathbb{C}\right),
$$
where Re and the bar denote the real part of a complex number and the complex conjugation, respectively. Moreover, we denote by $\|u\|_A$ the norm induced by this inner product.

Since $A \in L_{\text {loc }}^2(\mathbb{R}^N, \mathbb{R}^N)$, we have the following diamagnetic inequality  on $H_A^1(\mathbb{R}^N, \mathbb{C})$(see Theorem 7.21 in \cite{EHBM2001}):
\begin{equation}\label{eq2.1}
 \left|\nabla_A u(x)\right| \geq|\nabla| u(x)|| .
\end{equation}
Let
$$
E_\lambda:=\left\{u \in H_A^1(\mathbb{R}^N, \mathbb{C}): \int_{\mathbb{R}^N} \lambda V(x)|u|^2 d x<\infty\right\}
$$
with the norm
$$
\|u\|_\lambda^2=\int_{\mathbb{R}^N}\left(\left|\nabla_A u\right|^2+(\lambda V(x)+1)|u|^2\right) d x .
$$
Obviously, $(E_\lambda,\|\cdot\|_\lambda)$ is a Hilbert space and $E_\lambda \subset H_A^1(\mathbb{R}^N, \mathbb{C})$ for any $\lambda>0$.
Assume $B_R(0) \subset \mathbb{R}^N$ is an open set, consider
\begin{eqnarray*}
&& H_A^1(B_R(0)):=\left\{u \in L^2(B_R(0), \mathbb{C}):\left|\nabla_A u\right| \in L^2(B_R(0), \mathbb{R})\right\}, \\
&& \|u\|_{H_{A,R}^1}=\left(\int_{B_R(0)}\left(\left|\nabla_A u\right|^2+|u|^2\right) d x\right)^{\frac{1}{2}}, \\
&& E_{\lambda,R}(B_R(0), \mathbb{C}):=\left\{u \in H_A^1(B_R(0), \mathbb{C}): \int_{B_R(0)} \lambda V(x)|u|^2 d x<\infty\right\}, \\
&& \|u\|_{\lambda, R}^2=\int_{B_R(0)}\left(\left|\nabla_A u\right|^2+(\lambda V(x)+1)|u|^2\right) d x .
\end{eqnarray*}
Let $H_A^{0,1}(B_R(0), \mathbb{C})$ be the Hilbert space defined by the closure of $C_0^{\infty}(B_R(0), \mathbb{C})$ under the norm $\|u\|_{H_{A,R}^1}$. \eqref{eq2.1} implies that if $u\in H_A^1(\mathbb{R}^N, \mathbb{C})$, then $|u|\in H^1(\mathbb{R}^N, \mathbb{R})$. Therefore, the embedding $E_\lambda \hookrightarrow L^r(\mathbb{R}^N, \mathbb{C})$ is continuous for $2 \leq r \leq 2^*$ and the embedding $E_\lambda \hookrightarrow L_{\mathrm{loc}}^r(\mathbb{R}^N, \mathbb{C})$ is compact for $1 \leq r<2^*$.

From a mathematical perspective, The problem \eqref{eq1.1} has some interesting difficulties. In order to solve the problem \eqref{eq1.1}, if we try to apply variational method, the energy functional corresponding to \eqref{eq1.1} is
$$
I_{\lambda}(v)=\frac{1}{2} \int_{\mathbb{R}^N}\left(\left|\nabla_A v\right|^2+(\lambda V(x)+1)|v|^2\right) d x-\frac{1}{q}\int_{\mathbb{R}^N} |v|^{q} d x-\int_{\mathbb{R}^N} F(v) d x
$$
and
$$
\langle I'_{\lambda}(v),\phi\rangle= \mathrm{Re}\left(\int_{\mathbb{R}^N}\left(\nabla_A v\overline{\nabla_A\phi}+(\lambda V(x)+1)v\overline{\phi}\right) d x-\int_{\mathbb{R}^N} |v|^{q-2}v\overline{\phi} d x-\int_{\mathbb{R}^N} F'(v)\overline{\phi} d x\right)
$$
with
$$
F(v)=\int_0^v s \log |s|^2 d s=\frac{1}{2} |v|^2 \log |v|^2-\frac{|v|^2}{2}\ \text{and}\ \ v,\ \phi\in E_\lambda.
$$
According to the fact there exist functions $u \in H_A^1(\mathbb{R}^N, \mathbb{C})$ such that $\int_{\mathbb{R}^N} |u|^2 \log |u|^2dx=-\infty$, which gives the possibility that
$I_{\lambda}(u)=\infty$, then the functional $I_{\lambda}$ is not well defined on $H_A^1(\mathbb{R}^N, \mathbb{C})$. To overcome this difficulty, we consider a decomposition of the form
$$
F_2(t)-F_1(t)=\frac{1}{2} |t|^2 \log |t|^2, \  \forall t \in \mathbb{C},
$$
where $F_1\in C^1$ and $F_1$ is a nonnegative convex function, $F_2\in C^1$ satisfies Sobolev subcritical growth. This decomposition has been explored in many works (see \cite{{MSAS2015},{COACJ12020},{COACJ22023}}). In fact, fixed $\delta>0$ small enough, let
\begin{equation*}
  F_1(s):=\left\{\aligned
&0,  &s  =0, \\
&-\frac{1}{2}| s|^2 \log |s|^2, & 0<|s|  <\delta, \\
&-\frac{1}{2} |s|^2\left(\log \delta^2+3\right)+2 \delta|s|-\frac{\delta^2}{2}, & |s| \geq \delta
\endaligned
\right.
\end{equation*}
and
$$
F_2(s):= \begin{cases}0, & |s|<\delta, \\ \frac{1}{2} |s|^2 \log \left(\frac{|s|^2}{\delta^2}\right)+2 \delta|s|-\frac{3}{2} |s|^2-\frac{\delta^2}{2}, & |s| \geq \delta\end{cases}
$$
for every $s \in \mathbb{C}$. Hence,
\begin{equation}\label{eq2.2}
  F_2(s)-F_1(s)=\frac{1}{2} s^2 \log s^2, \quad \forall s \in \mathbb{R} .
\end{equation}
It is easy to see that $F_1$ and $F_2$ satisfy the following properties:
\begin{itemize}
\item[$(f_1)$] $F_1$ is an even function with $F_1^{\prime}(s) s \geq 0$ and $F_1 \geq 0$. Moreover, $F_1 \in C^1(\mathbb{R}, \mathbb{R})$  is convex for $\delta \approx 0^{+}$.

\item[$(f_2)$] $F_2 \in C^1(\mathbb{R}, \mathbb{R}) \cap C^2((\delta,+\infty), \mathbb{R})$ and there exists $C=C_p>0$ such that
$$
\left|F_2^{\prime}(s)\right| \leq C|s|^{p-1}, \quad \forall s \in \mathbb{R},\ p \in\left(2,2^*\right).
$$

\item[$(f_3)$] $s \mapsto \frac{F_2^{\prime}(s)}{s}$ is a nondecreasing function for $s>0$ and a strictly increasing function for $s>\delta$.

\item[$(f_4)$] $\lim\limits_{s \rightarrow \infty} \frac{F_2^{\prime}(s)}{s}=\infty$.
\end{itemize}
Therefore, we can rewrite the functional $I_{\lambda}$ as
\begin{eqnarray*}
I_{\lambda}(u)=\frac{1}{2} \int_{\mathbb{R}^N}\left(\left|\nabla_A u\right|^2+(\lambda V(x)+1)|u|^2\right) d x-\frac{1}{q}\int_{\mathbb{R}^N} |u|^{q} d x +\int_{\mathbb{R}^N} F_1(u) d x-\int_{\mathbb{R}^N} F_2(u) d x.
\end{eqnarray*}
This technology guarantees that $I_{\lambda}$ can be decomposed as a sum of a $C^1$-functional with a convex and lower semi-continuous functional. In this way, the critical point theory of functionals developed in \cite{AS1986} can be used to find solutions of the problem \eqref{eq1.1}.

By using $(V_1)$, for any open set $K \subset \mathbb{R}^N$, it is easy to see that
$$
\|u\|_{2, K}^2 \leq \int_K\left(\left|\nabla_A u\right|^2+(\lambda V(x)+1)|u|^2\right) d x \quad \text { for all } u \in E_\lambda(K, \mathbb{C}) \text { and } \lambda>0
$$
where $\|u\|_{2, K}^2=\int_K|u|^2 d x$. Hence, similar to \cite{YHDKT2003}, it follows that
\begin{lemma}{\rm (see \cite[Corollary 1.4]{{YHDKT2003}})}\label{L2.1}
There exist $a_0, b_0>0$ with $a_0 \sim 1$ and $b_0 \sim 0$ such that for any open set $K \subset \mathbb{R}^N$,
\begin{equation*}
   a_0\|u\|_{\lambda, K}^2 \leq\|u\|_{\lambda, K}^2-b_0\|u\|_{2, K}^2,
\end{equation*}
for all $u \in E_\lambda(K, \mathbb{C})$ and $\lambda>0$.
 \end{lemma}
\section{Auxiliary problem}
For any $j \in\{1, \ldots, k\}$, fix a bounded open subset $\Omega_j^{\prime}$ with smooth boundary such that
\begin{equation*}
 \overline{\Omega_j} \subset \Omega_j^{\prime}
\end{equation*}
and
\begin{equation*}
  \overline{\Omega_j^{\prime}} \cap \overline{\Omega_l^{\prime}}=\emptyset\  \text{for all}\  j \neq l.
\end{equation*}
Next, fix a non-empty subset $\Gamma \subset\{1, \ldots, k\}$ and $R>0$ such that $\Omega_{\Gamma}^{\prime} \subset B_R(0)$ and
$$
\Omega_{\Gamma}=\bigcup_{j \in \Gamma} \Omega_j, \quad \Omega_{\Gamma}^{\prime}=\bigcup_{j \in \Gamma} \Omega_j^{\prime} .
$$

From now on, we will study the auxiliary problem of \eqref{eq1.1}. Fix $b_0 \approx 0^{+}$ and $a_0>\delta$ are  constants given in Lemma \ref{L2.1} in a such way that $F_2^{\prime}(s) +a_0^{q-2}s=b_0s$ with $|s|=a_0$. Using these notations, we set
\begin{equation*}
   \widetilde{F}_2^{\prime}(s):=\left\{\aligned
&F_2^{\prime}(s)+|s|^{q-2}s, & 0 \leq |s| \leq a_0, \\
&b_0 s, & |s| \geq a_0
\endaligned
\right.
\end{equation*}
and
$$
g_2(x, t):=\chi_{\Gamma}(x) F_2^{\prime}(t)+\left(1-\chi_{\Gamma}(x)\right) \widetilde{F}_2^{\prime}(t),
$$
where
$$
\chi_{\Gamma}(x):=\left\{\begin{array}{ll}
1, &  x \in \Omega_{\Gamma}^{\prime}, \\
0, & x \in B_R(0)\backslash\Omega_{\Gamma}^{\prime}.
\end{array}\right.
$$
Since then, we will study the existence of solution for the following auxiliary problem
\begin{equation}\label{eq3.1}
-(\nabla+i A(x))^2 u+(\lambda V(x)+1) u=g_2(x, u )-F_1^{\prime}(u), \ x\in B_R(0),
\end{equation}
and $u = 0$ on $\partial B_R(0)$. It is worth mentioning that, if $u_{\lambda, R}$ is a solution of \eqref{eq3.1} with $0<|u_{\lambda, R}| \leqslant a_0$ for all $x \in B_R(0) \backslash \Omega_{\Gamma}^{\prime}$, then $g_2\left(x, u_{\lambda, R}\right)=F_2^{\prime}\left(u_{\lambda, R}\right)+|u_{\lambda, R}|^{q-2}u_{\lambda, R}$, so $u_{\lambda, R}$ is also a solution of
\begin{equation}\label{eq3.2}
  \begin{cases}-(\nabla+i A(x))^2 u+\lambda V(x) u=|u|^{q-2}u+u \log |u|^2, & \text { in } B_R(0),  \\ u=0, & \text { on } \partial B_R(0) .\end{cases}
\end{equation}
Therefore, our goal is to find the nontrivial critical points for the functional
$$
I_{\lambda, R}(u)=\int_{B_R(0)}\left(\left|\nabla_A u\right|^2+(\lambda V(x)+1)|u|^2\right) d x+\int_{B_R(0)} F_1(u) d x-\int_{B_R(0)} G_2(x, u ) d x,
$$
where $G_2(x, t)=\int_0^t g_2(x, s) d s$ for all $(x, t) \in B_R(0) \times \mathbb{C}$. Clearly, $I_{\lambda, R} \in C^1(E_{\lambda, R}, \mathbb{R})$.

Now, we show that $I_{\lambda, R}$ satisfies the geometric structure of the Mountain Pass Theorem.
\begin{lemma}\label{L3.1}
For all $\lambda>0$, the functional $I_{\lambda, R}$ satisfies the following conditions:

$\mathrm{(i)}$ there exist $\alpha, \rho>0$ such that $I_{\lambda, R}(u) \geqslant \alpha$ for any $u \in E_{\lambda, R}$ with $\|u\|_{\lambda, R}=\rho$;

$(\mathrm{ii)}$ there exists $e \in E_{\lambda, R}$ with $\|e\|_{\lambda, R}>\rho$ such that $I_{\lambda, R}(e)<0$.
\end{lemma}
\begin{proof}
(i) Using $(f_1)-(f_2)$, it follows that
\begin{eqnarray*}
I_{\lambda, R}(u)&=&\frac{1}{2}\int_{B_R(0)}\left(\left|\nabla_A u\right|^2+(\lambda V(x)+1)|u|^2\right) d x+\int_{B_R(0)} F_1(u) d x-\int_{B_R(0)} G_2(x, u) d x\\
&\geq&\frac{1}{2}\|u\|_{\lambda, R}^2-\int_{B_R(0)} G_2(x, u) d x\\
&\geq&\frac{1}{2}\|u\|_{\lambda, R}^2-\int_{B_R(0)} F_2(u) d x-\int_{B_R(0)}|u|^q d x\\
&\geq&\frac{1}{2}\|u\|_{\lambda, R}^2-C\|u\|_{\lambda, R}^p-C_1\|u\|_{\lambda, R}^q\\
&>&0,
\end{eqnarray*}
where $C_1>0$ and $\|u\|_{\lambda, R}>0$ small enough.

(ii) Fixing $0<w \in C_0^{\infty}(\Omega_{\Gamma})$, by \eqref{eq2.2}, we have
\begin{eqnarray*}
I_{\lambda, R}(sw)&=&\frac{s^2}{2}\|w\|_{\lambda, R}^2+\int_{B_R(0)} F_1(sw) d x-\int_{B_R(0)} G_2(x, sw) d x\\
&=&s^2I_{\lambda, R}(w)+\int_{B_R(0)} [F_1(sw)-F_1(w)] d x+\int_{B_R(0)}[G_2(x, u)-G_2(x, sw)] d x   \\
&=&s^2I_{\lambda, R}(w)+\int_{B_R(0)} \left(\frac{1}{2}|w|^2\log |w|^2-\frac{s^2}{2}|w|^2\log s^2|w|^2\right)dx\\
&=& \int_{B_R(0)} \frac{1}{2}|w|^2\log |w|^2dx+s^2\left(I_{\lambda, R}(w)-\int_{B_R(0)} \frac{1}{2}|w|^2\log s^2|w|^2dx\right)\\
&\rightarrow&-\infty
\end{eqnarray*}
as $s \rightarrow+\infty$. Hence, there exists $s_0>0$ independent of $\lambda>0$ and $R>0$ large enough such that $I_{\lambda, R}(s_0w)<0$, that is $e=s_0w$.
\end{proof}

By Lemma \ref{L3.1} and using a variant of the mountain pass theorem without the Palais-Smale condition (see Theorem 2.9 in \cite{MW1996}), we obtain the mountain pass level associated with $I_{\lambda, R}$, denoted by $c_{\lambda, R}$, is given by
$$
c_{\lambda, R}=\inf\limits_{\gamma \in \Gamma_{\lambda, R}} \max _{t \in[0,1]} I_{\lambda, R}(\gamma(t)),
$$
where $\Gamma_{\lambda, R}=\left\{\gamma \in C\left([0,1], E_{\lambda, R}\right): \gamma(0)=0\right.$ and $\left.I_{\lambda, R}(\gamma(1))<0\right\}$. Moreover, by Lemma \ref{L3.1},
$$
c_{\lambda, R} \geqslant \alpha>0, \quad \forall \lambda>0, \quad R>0 \text { large enough. }
$$

In the sequel, the following logarithmic inequality is useful(see \cite[page 153]{{MDL2003}})
$$
\int_{\mathbb{R}^N}|u|^2 \log \left(\frac{|u|}{\|u\|_2}\right) \leq C\|u\|_2 \log \left(\frac{\|u\|_{2^*}}{\|u\|_2}\right), \quad \forall u \in L^2\left(\mathbb{R}^N\right) \cap L^{2^*}(\mathbb{R}^N)
$$
for some positive constant $C$. According to this inequality, it is obvious that
\begin{equation}\label{eq3.3}
  \int_{\Lambda_{\varepsilon}}|u|^2 \log \left(\frac{|u|}{\|u\|_{L^2\left(\Lambda_{\varepsilon}\right)}}\right) \leq C\|u\|_{L^2\left(\Lambda_{\varepsilon}\right)} \log \left(\frac{\|u\|_{L^{2^*}\left(\Lambda_{\varepsilon}\right)}}{\|u\|_{L^2\left(\Lambda_{\varepsilon}\right)}}\right), \quad \forall u \in L^2\left(\Lambda_{\varepsilon}\right) \cap L^{2^*}\left(\Lambda_{\varepsilon}\right) .
\end{equation}
\begin{lemma}\label{L3.2}
Assume $\left(v_n\right)$ is a $(P S)_c$ sequence for $I_{\lambda, R}$, the sequence $\left(v_n\right)$ is bounded in $E_{\lambda, R}$.
\end{lemma}
\begin{proof}
Due to $\left(v_n\right)$ is $(P S)_c$ sequence for $I_{\lambda, R}$, it is obvious that
\begin{equation}\label{eq3.4}
  I_{\lambda, R}\left(v_n\right)-\frac{1}{2}\langle I_{\lambda, R}^{\prime}\left(v_n\right), v_n\rangle\leq c+1 +o_n(1)\left\|v_n\right\|_{\lambda, R}
\end{equation}
for large $n$. Note that,
$$
\int_{B_R(0)}\left[\left(F_1\left(v_n\right)-\frac{1}{2} F_1^{\prime}\left(v_n\right) \overline{v_n}\right)+\left(\frac{1}{2} F_2^{\prime}\left(v_n\right) \overline{v_n}-F_2\left(v_n\right)\right)\right]dx= \frac{1}{2}\int_{B_R(0)}\left|v_n\right|^2dx ,
$$
so we have
\begin{eqnarray}\label{eq3.5}
&& I_{\lambda, R}\left(v_n\right)-\frac{1}{2}\langle I_{\lambda, R}^{\prime}\left(v_n\right), v_n\rangle \nonumber\\
&=&\frac{1}{2}\|v_n\|_{\lambda, R}^2+\int_{B_R(0)} F_1(v_n) d x-\int_{B_R(0)} G_2(x, v_n ) d x-\frac{1}{2}\|v_n\|_{\lambda, R}^2-\frac{1}{2}\int_{B_R(0)} F'_1(v_n)\overline{v_n}dx\nonumber\\
&&+\frac{1}{2}\int_{B_R(0)} G'_2( x, v_n )\overline{v_n}dx\nonumber\\
& =&\int_{B_R(0)}\left(F_1\left(v_n\right)-\frac{1}{2} F_1^{\prime}\left(v_n\right) \overline{v_n}\right)dx+\int_{B_R(0)}\left(\frac{1}{2}  G'_2( x, v_n )\overline{v_n}-G_2\left( x, v_n \right)\right)dx \nonumber\\
& =&\frac{1}{2}\int_{B_R(0)}\left|v_n\right|^2dx+\int_{B_R(0)}\left(F_2\left(v_n\right)-\frac{1}{2} F_2^{\prime}\left(v_n\right) \overline{v_n}+\frac{1}{2}  G'_2( x, v_n )\overline{v_n}-G_2\left( x, v_n \right)\right)dx\nonumber\\
&=&\frac{1}{2} \int_{\Omega_{\Gamma}^{\prime}}\left|v_n\right|^2dx+\int_{B_R(0)\backslash\Omega_{\Gamma}^{\prime}}\left(\frac{1}{2}\left|v_n\right|^2+F_2\left(v_n\right)-\frac{1}{2} F_2^{\prime}\left(v_n\right) \overline{v_n}\right)dx\nonumber\\
&&+\int_{B_R(0)\backslash\Omega_{\Gamma}^{\prime}}\left(\frac{1}{2} G_2^{\prime}\left( x, v_n \right) \overline{v_n}-G_2\left( x, v_n \right)\right)dx.
\end{eqnarray}
By using the fact
\begin{equation*}
  \frac{1}{2}|t|^2+\left[F_2(t)-\frac{1}{2} F_2^{\prime}(t) \overline{t}+\frac{1}{2} G_2^{\prime}( x, t) \overline{t}-G_2( x, t)\right] \geq 0,\   t \in \mathbb{C},\ x \in \mathbb{R}^N,
\end{equation*}
we get
\begin{equation*}
  I_{\lambda, R}\left(v_n\right)-\frac{1}{2}\langle I_{\lambda, R}^{\prime}\left(v_n\right), v_n\rangle \geq \frac{1}{2} \int_{\Omega_{\Gamma}^{\prime}}\left|v_n\right|^2dx,
\end{equation*}
so \eqref{eq3.4} implies that
\begin{equation}\label{eq3.6}
(c+1)+o_n(1)\left\|v_n\right\|_{\varepsilon} \geq \frac{1}{2} \int_{\Omega_{\Gamma}^{\prime}}\left|v_n\right|^2dx.
\end{equation}
According to the definition of $F_1$, there are constants $A, B>0$ satisfying
$$
F_1(t) \leq A|t|^2+B, \quad \forall t \in \mathbb{R} .
$$
This together with \eqref{eq3.6} leads to
\begin{equation}\label{eq3.7}
  \int_{B_R(0)} F_1\left(v_n\right)dx \leq C_{\varepsilon}+o_n(1)\left\|v_n\right\|_{\lambda, R}
\end{equation}
for some $C_{\varepsilon}>0$. Thanks to \eqref{eq3.3}, we obtain that
\begin{eqnarray*}
&&\frac{1}{2} \int_{B_R(0)}\left|v_n\right|^2 \log \left|v_n\right|^2dx=\int_{B_R(0)}\left|v_n\right|^2 \log \left|v_n\right|dx \\
&=&\int_{B_R(0)}\left|v_n\right|^2 \log \left(\frac{\left|v_n\right|}{\left\|v_n\right\|_{L^2\left(B_R(0)\right)}}\right)dx+\left\|v_n\right\|_{L^2\left(B_R(0)\right)}^2 \log \left(\left\|v_n\right\|_{L^2\left(B_R(0)\right)}\right) \\
&\leq& C\left\|v_n\right\|_{L^2\left(B_R(0)\right)} \log \left(\frac{\left\|v_n\right\|_{L^{2^*}\left(B_R(0)\right)}}{\left\|v_n\right\|_{L^2\left(B_R(0)\right)}}\right)+\left\|v_n\right\|_{L^2\left(B_R(0)\right)}^2 \log \left(\left\|v_n\right\|_{L^2\left(B_R(0)\right)}\right)\\
&=&\left(\left\|v_n\right\|_{L^2\left(B_R(0)\right)}^2-C\left\|v_n\right\|_{L^2\left(B_R(0)\right)}\right) \log \left(\left\|v_n\right\|_{L^2\left(B_R(0)\right)}\right)+C\left\|v_n\right\|_{L^2\left(B_R(0)\right)} \log \left(\left\|v_n\right\|_{L^{2^*}\left(B_R(0)\right)}\right),
\end{eqnarray*}
then combines with the embedding $E_{\lambda, R} \hookrightarrow L^t\left(B_R(0)\right)$ to give
\begin{eqnarray*}
\int_{B_R(0)}\left|v_n\right|^2 \log \left|v_n\right|^2dx &\leq&\left(2\| v_n\|_{L^2(B_R(0))}^2-2 C\| v_n\|_{L^2(B_R(0))}\right) \log \left(|| v_n \|_{L^2\left(B_R(0)\right)}\right)\\
&&+ \widetilde{C}\left\|v_n\right\|_{\lambda, R} \left|\log ( \widetilde{C}\left\|v_n\right\|_{\lambda, R})\right|
\end{eqnarray*}
for some convenient $\widetilde{C}>0$ independent of $\varepsilon$. In the last inequality, we use the fact that the function $t \mapsto \log t$ is increasing on $t>0$. Now, using the fact that given $r \in(0,1)$, there is $A>0$ satisfying
$$
|t \log t| \leq A\left(1+|t|^{1+r}\right), \quad t \geq 0.
$$
According to \eqref{eq3.6}, we have
$$
\left\|v_n\right\|_{L^2\left(B_R(0)\right)} \log \left(\left\|v_n\right\|_{L^2\left(B_R(0)\right)}\right) \leq A\left(1+\left\|v_n\right\|_{L^2\left(B_R(0)\right)}^{1+r}\right)
$$
and
$$
\left.\left\|v_n\right\|_{L^2\left(B_R(0)\right)}^2 \log \left(\left\|v_n\right\|_{L^2\left(B_R(0)\right)}^2\right) \leq A\left(1+\left(\left\|v_n\right\|_{L^2\left(B_R(0)\right)}^2\right)^{1+r}\right) \leq \tilde{A}\left(1+\left\|v_n\right\|_{L^2\left(B_R(0)\right)}\right)^{1+r}\right) .
$$
Hence, modifying $A$ if necessary, it holds
$$
\int_{B_R(0)}\left|v_n\right|^2 \log \left|v_n\right|^2dx \leq A\left(1+\left\|v_n\right\|_{\lambda, R}^{1+r}\right) .
$$
As $\left(v_n\right)$ is $(P S)_c$ sequence,
\begin{eqnarray*}
(c+1) &\geq& I_{\lambda, R}\left(v_n\right)\\
&\geq&\frac{1}{2}\left\|v_n\right\|_{\lambda, R}^2+\int_{B_R(0)} F_1(v_n)dx-\int_{B_R(0)} G_2( x, v_n)dx\\
&=&\frac{1}{2}\left\|v_n\right\|_{\lambda, R}^2+\int_{B_R(0)\backslash\Omega_{\Gamma}^{\prime}} F_1\left(v_n\right)dx-\frac{1}{2}\int_{\Omega_{\Gamma}}\left|v_n\right|^2 \log \left|v_n\right|^2dx\\
&&-\int_{B_R(0)\backslash\Omega_{\Gamma}^{\prime}}  G_2\left(  x, v_n\right)dx
\end{eqnarray*}
for large $n$. By using the definition of $\widetilde{F}_2^{\prime}(s)$,
$$
G_2( x, t) \leq \frac{b_0}{2} t^2, \quad \forall x \in B_R(0)\backslash\Omega_{\Gamma}^{\prime},
$$
then
$$
(c+1)+A\left(1+\left\|v_n\right\|_{\lambda, R}^{1+r}\right) \geq C\left\|v_n\right\|_{\lambda, R}^2+\int_{B_R(0)\backslash\Omega_{\Gamma}^{\prime}} F_1\left(v_n\right)dx,
$$
for some $C>0$. Since $r \in(0,1)$, so $\left(v_n\right)$ is bounded in $E_{\lambda, R}$.
\end{proof}
\begin{lemma}\label{L3.3}
The functional $I_{\lambda, R}$ satisfies the $(PS)$ condition.
\end{lemma}
\begin{proof}
Assume that $\left(u_n\right) \subset E_{\lambda, R}$ is a $(\mathrm{PS})_c$ sequence for $I_{\lambda, R}$, that is,
$$
I_{\lambda, R}\left(u_n\right) \rightarrow c \quad \text { and } \quad I_{\lambda, R}^{\prime}\left(u_n\right) \rightarrow 0.
$$
By Lemma \ref{L3.2}, the sequence $\left(u_n\right)$ is bounded in $E_{\lambda, R}$. Without loss of generality, we always assume that there exist $u \in E_{\lambda, R}$ and a subsequence $\left(u_n\right)$ such that
\begin{eqnarray*}
u_n &\rightarrow u & \text { in } E_{\lambda, R}, \\
u_n &\rightarrow u & \text { in } L^t(B_R(0)),\ \forall t \in\left[1,2^*\right)\\
u_n &\rightarrow u  & \text { a.e. in } B_R(0) .
\end{eqnarray*}
Fix $p \in\left(2,2^*\right)$, we see that there exists $C_1,\ C_2>0$ such that
$$
\left|G_2^{\prime}(x, t)\right| \leq  C_1|t|+C_2|t|^{p-1},\ \forall t \in \mathbb{R}
$$
and
$$
\left|F_1^{\prime}(t)\right| \leq  C\left(1+|t|^{p-1}\right),\ \forall t \in \mathbb{R} .
$$
Therefore, by the Sobolev embedding theorem,
\begin{equation*}
  \int_{B_R(0)} G_2^{\prime}\left(x, u_n \right) \overline{u_n } d x  \rightarrow \int_{B_R(0)} G_2^{\prime}\left(x, u \right) \overline{u } d x ,
\end{equation*}
\begin{equation*}
 \int_{B_R(0)} F_1^{\prime}\left(u_n\right) \overline{u_n} d x  \rightarrow \int_{B_R(0)} F_1^{\prime}(u) \overline{u} d x  ,
\end{equation*}
\begin{equation*}
  \mathrm{Re}\left(\int_{B_R(0)} G_2^{\prime}\left(x, u_n \right) \overline{v} d x\right) \rightarrow \mathrm{Re}\left(\int_{B_R(0)} G_2^{\prime}\left(x, u \right) \overline{v} d x\right),
\end{equation*}
and
\begin{equation*}
  \mathrm{Re}\left(\int_{B_R(0)} F_1^{\prime}\left(u_n\right) \overline{v} d x\right) \rightarrow \mathrm{Re}\left(\int_{B_R(0)} F_1^{\prime}(u) \overline{v} d x\right)
\end{equation*}
for any $v \in E_{\lambda, R}$.
Now, using the fact that
\begin{eqnarray*}
&&\langle I_{\lambda, R}^{\prime}\left(u_n\right), u_n-u\rangle\\
&=&\mathrm{Re}\left( \int_{B_R(0)}\left(\nabla_A u_n\overline{\nabla_A (u_n-u)}+(\lambda V(x)+1)u_n\overline{u_n-u}\right) d x\right)\\
&&+\mathrm{Re}\left(\int_{B_R(0)} F'_1(u_n)\overline{u_n-u} d x\right)-\mathrm{Re}\left(\int_{B_R(0)} G'_2(x, u_n )\overline{u_n -u } d x\right)
\end{eqnarray*}
and
\begin{eqnarray*}
&&\langle I_{\lambda, R}^{\prime}\left(u\right), u_n-u\rangle\\
&=&\mathrm{Re}\left( \int_{B_R(0)}\left(\nabla_A u\overline{\nabla_A (u_n-u)}+(\lambda V(x)+1)u\overline{u_n-u}\right) d x\right)\\
&&+\mathrm{Re}\left(\int_{B_R(0)} F'_1(u)\overline{u_n-u} d x\right)-\mathrm{Re}\left(\int_{B_R(0)} G'_2(x, u )\overline{u_n -u } d x\right),
\end{eqnarray*}
so we obtain
\begin{eqnarray*}
\left\|u_n-u\right\|_{\lambda, R}^2&= & \mathrm{Re}\left(\int_{B_R(0)}\left(G_2^{\prime}\left(x, u_n \right)-G_2^{\prime}\left(x, u \right)\right)\overline{u_n -u } d x\right) \\
& &-\mathrm{Re}\left(\int_{B_R(0)}\left(F_1^{\prime}\left(u_n\right)-F_1^{\prime}(u)\right)\overline{u_n-u} d x\right)+o_n(1)=o_n(1).
\end{eqnarray*}
The proof is completed.
\end{proof}
\begin{lemma}\label{L3.4}
The problem \eqref{eq3.1} has a  nontrivial solution $u_{\lambda, R} \in E_{\lambda, R}$ such that $I_{\lambda, R}\left(u_{\lambda, R}\right)=c_{\lambda, R}$, where $c_{\lambda, R}$ denotes the mountain pass level associated with $I_{\lambda, R}$.
\end{lemma}
\begin{proof}
The existence of the nontrivial solution $u_{\lambda, R}$ is an immediate result of Lemmas \ref{L3.1} and \ref{L3.3}.
\end{proof}
Next, for each $R>0$, we will consider the behavior of a $(PS)_{\infty, R}$ sequence for $I_{\lambda, R}$, that is, a sequence $\left(u_n\right) \subset H_0^1\left(B_R(0)\right)$ satisfying
$$
\begin{aligned}
& u_n \in E_{\lambda_n, R} \text { and } \lambda_n \rightarrow \infty, \\
& I_{\lambda_n, R}\left(u_n\right) \rightarrow c,\ \left\|I_{\lambda_n, R}^{\prime}\left(u_n\right)\right\| \rightarrow 0.
\end{aligned}
$$

\begin{lemma}\label{L3.5}
Let $\left(u_n\right) \subset H_A^1\left(B_R(0)\right)$ be a $(\mathrm{PS})_{\infty, R}$ sequence. Then for some subsequence $\left(u_n\right)$, there exists $u \in H_A^1(B_R(0))$ such that
$$
u_n \rightharpoonup u \quad \text { in } H_A^1(B_R(0)) .
$$
Moreover,

$\mathrm{(i)}$ $\left\|u_n-u\right\|_{\lambda_n, R} \rightarrow 0$,
and hence,
$$
u_n \rightarrow u \quad \text { in } H_A^1(B_R(0)).
$$

$\mathrm{(ii)}$ $u \equiv 0$ in $B_R(0) \backslash \Omega_{\Gamma}$ and $u$ is a solution of
\begin{equation}\label{eq3.8}
  \begin{cases}-(\nabla+iA(x))^2u=|u|^{q-2}u+u \log |u|^2 & \text { in } \Omega_{\Gamma}, \\ u=0 & \text { on } \partial \Omega_{\Gamma}.\end{cases}
\end{equation}

$\mathrm{(iii)}$ $u_n$ also satisfies
$$
\begin{aligned}
& \lambda_n \int_{B_R(0)} V(x)\left|u_n\right|^2 d x \rightarrow 0, \\
& \left\|u_n\right\|_{\lambda_n, B_R(0) \backslash \Omega_{\Gamma}}^2 \rightarrow 0, \\
& \left\|u_n\right\|_{\lambda_n, \Omega_j^{\prime}}^2 \rightarrow \int_{\Omega_j}\left(|\nabla_A u|^2+|u|^2\right) d x \quad \text { for all } j \in \Gamma .
\end{aligned}
$$
\end{lemma}
\begin{proof}
By using Lemma \ref{L3.2}, there exists $K>0$ satisfying
$$
\left\|u_n\right\|_{\lambda_n, R}^2 \leqslant K, \quad \forall n \in \mathbb{N} .
$$
Thus, $\left(u_n\right)$ is bounded in $H_A^1(B_R(0))$ and we can assume that for some $u \in H_A^1(B_R(0))$,
$$
u_n \rightharpoonup u \ \text {in}\ H_A^1\left(B_R(0)\right)
$$
and
$$
u_n(x) \rightarrow u(x) \text { a.e. in } B_R(0) .
$$
Fixing $C_m=\left\{x \in B_R(0): V(x) \geqslant \frac{1}{m}\right\}$, we have
$$
\int_{C_m}\left|u_n\right|^2 d x \leqslant \frac{m}{\lambda_n} \int_{B_R(0)} \lambda_n V(x)\left|u_n\right|^2 d x,
$$
so
$$
\int_{C_m}\left|u_n\right|^2 d x \leqslant \frac{m}{\lambda_n}\left\|u_n\right\|_{\lambda_n, R}^2.
$$
Using Fatou's lemma, it holds
$$
\int_{C_m}|u|^2 d x=0, \quad \forall m \in \mathbb{N}.
$$
Then $u(x)=0$ on $\bigcup\limits_{m=1}^{+\infty} C_m=B_R(0) \backslash \overline{\Omega}$, so $\left.u\right|_{\Omega_j} \in H_A^1\left(\Omega_j\right), j \in\{1, \ldots, k\}$. Next, we will prove (i)-(iii).

(i) Note that $u=0$ in $B_R(0) \backslash \overline{\Omega}$ and $\langle I_{\lambda, R}^{\prime}\left(u_n\right), u_n-u\rangle=\langle I_{\lambda, R}^{\prime}\left(u\right), u_n-u\rangle=o_n(1)$, similar to the proof of lemma \ref{L3.3}, it holds
$$
\int_{B_R(0)}\left(\left|\nabla_A(u_n-u)\right|^2+\left(\lambda_n V(x)+1\right)\left|u_n-u\right|^2\right) d x\rightarrow 0,
$$
so $u_n \rightarrow u$ in $H_A^1\left(B_R(0)\right)$.

(ii) On the one hand, recalling that $u \in H_A^1\left(B_R(0)\right)$ and $u=0$ in $B_R(0) \backslash \overline{\Omega}$, we get that $u \in H_A^1(\Omega)$ or $\left.u\right|_{\Omega_j}$ $\in H_A^1\left(\Omega_j\right)$ for $j=1, \ldots, k$. Moreover, $u_n \rightarrow u$ in $H_A^1\left(B_R(0)\right)$ combined with $\langle I_{\lambda_n, R}^{\prime}\left(u_n\right), \varphi\rangle \rightarrow 0$ as $n \rightarrow+\infty$ for each $\varphi \in C_0^{\infty}\left(\Omega_{\Gamma}\right)$ implies that
$$
\mathrm{Re}\left(\int_{\Omega_{\Gamma}}(\nabla_A u \overline{\nabla_A \varphi}+u \overline{\varphi}) d x+\int_{\Omega_{\Gamma}} F_1^{\prime}(u) \overline{\varphi} d x-\int_{\Omega_{\Gamma}} F_2^{\prime}\left(u\right) \overline{\varphi} d x\right)=0,
$$
so $\left.u\right|_{\Omega_{\Gamma}}$ is a solution for \eqref{eq3.8}.
On the other hand, for each $j \in\{1,2, \ldots, k\} \backslash \Gamma$, it follows that
$$
\int_{\Omega_j}\left(|\nabla_A u|^2+u^2\right) d x+\mathrm{Re}\left(\int_{\Omega_j} F_1^{\prime}(u) \overline{u} d x-\int_{\Omega_j} \widetilde{F}_2^{\prime}(u) \overline{u} d x\right)=0 .
$$
By the fact that $F_1^{\prime}(s) s \geqslant 0$ and $\widetilde{F}_2^{\prime}(s) s \leq  b_0 s^2$ for all $s \in \mathbb{R}^{+}$, we obtain that
$$
\int_{\Omega_j}\left(|\nabla_A u|^2+u^2\right) d x \leq \mathrm{Re}\left(\int_{\Omega_j} \widetilde{F}_2^{\prime}(u) \overline{u} d x\right) \leq  b_0 \int_{\Omega_j} u^2dx .
$$
Due to $b_0<1, u=0$ in $\Omega_j$ for $j \in\{1,2, \ldots, k\} \backslash \Gamma$ and $u \geq  0$ in $B_R(0)$, which implies (ii) holds.

In order to get (iii), it follows from (i) that
$$
\int_{B_R(0)} \lambda_n V(x)\left|u_n\right|^2 d x=\int_{B_R(0)} \lambda_n V(x)\left|u_n-u\right|^2 d x \leqslant C\left\|u_n-u\right\|_{\lambda_n, R}^2,
$$
so
\begin{equation}\label{eq3.9}
  \int_{B_R(0)} \lambda_n V(x)\left|u_n\right|^2 d x \rightarrow 0 \quad \text { as } n \rightarrow+\infty.
\end{equation}
Moreover, using (i) and (ii), it is obviously that
$$
\begin{aligned}
& \left\|u_n\right\|_{\lambda_n, B_R(0) \backslash \Omega_{\Gamma}}^2 \rightarrow 0, \\
& \left\|u_n\right\|_{\lambda_n, \Omega_j^{\prime}}^2 \rightarrow \int_{\Omega_j}\left(|\nabla_A u|^2+|u|^2\right) d x \quad \text { for all } j \in \Gamma .
\end{aligned}
$$
In fact, note that  $u=0$ in $B_R(0) \backslash \overline{\Omega}$ and $u=0$ in $\Omega_j$ for $j \in\{1,2, \ldots, k\} \backslash \Gamma$. Since $u_n \rightarrow u$ in $H_A^1(B_R(0))$, we have $\left\|u_n\right\|_{\lambda_n, B_R(0) \backslash \Omega_{\Gamma}}^2 \rightarrow 0$. Because of \eqref{eq3.9}, the last conclusion holds.
\end{proof}
With a few modifications in the arguments in the proof of Lemma \ref{L3.5} and using Lemma \ref{L3.2}, we also have the following result, which is useful in Section 3.

\begin{lemma}\label{L3.6}
Let $u_n \in E_{\lambda_n, R_n}$ be a $(\mathrm{PS})_{\infty, R_n}$ sequence with $R_n \rightarrow+\infty$, that is,
\begin{equation*}
  u_n \in E_{\lambda_n, R_n} \  \text {and}\ \lambda_n \rightarrow \infty,\  I_{\lambda_n, R_n}\left(u_n\right) \rightarrow c,\ \left\|I_{\lambda_n, R_n}^{\prime}\left(u_n\right)\right\| \rightarrow 0 .
\end{equation*}
Then for some subsequence, still denoted by $(u_n)$, there exists $u \in H_A^1(\mathbb{R}^N)$ such that
$$
u_n \rightharpoonup u \quad \text { in } H_A^1(\mathbb{R}^N).
$$
Moreover,

$\mathrm{(i)}$ $\left\|u_n-u\right\|_{\lambda_n, R_n} \rightarrow 0$, so
$$
u_n \rightarrow u \quad \text { in }  H_A^1(\mathbb{R}^N).
$$

$\mathrm{(ii)}$ $u \equiv 0$ in $\mathbb{R}^N \backslash \Omega_{\Gamma}$ and $u$ is a solution of
$$
\begin{cases}-(\nabla+iA(x))^2u=|u|^{q-2}u+u \log u^2 & \text { in } \Omega_{\Gamma}, \\ u=0 & \text { on } \partial \Omega_{\Gamma}\end{cases}
$$

$\mathrm{(iii)}$ $u_n$ also satisfies
$$
\begin{aligned}
& \lambda_n \int_{B_{R_n}(0)} V(x)\left|u_n\right|^2 d x \rightarrow 0, \\
& \left\|u_n\right\|_{\lambda_n, B_{R_n}(0) \backslash \Omega_{\Gamma}}^2 \rightarrow 0, \\
& \left\|u_n\right\|_{\lambda_n, \Omega_j^{\prime}}^2 \rightarrow \int_{\Omega_j}\left(|\nabla u|^2+|u|^2\right) d x \quad \text { for all } j \in \Gamma .
\end{aligned}
$$
\end{lemma}
\begin{proof}
First of all, the boundedness of $\left(I_{\lambda_n, R_n}\left(u_n\right)\right)$ implies that there exists $K>0$ satisfying
$$
\left\|u_n\right\|_{\lambda_n, R_n}^2 \leq  K, \quad \forall n \in \mathbb{N} .
$$
So we can assume that there exists $u \in H_A^1(\mathbb{R}^N)$ satisfying
$$
u_n \rightharpoonup u \quad \text { in } H_A^1(\mathbb{R}^N)
$$
and
$$
u_n(x) \rightarrow u(x) \text { a.e. in } \mathbb{R}^N,
$$
and $u(x)=0$ on $\mathbb{R}^N \backslash \overline{\Omega}$.

(i) Let $0<R<R_n$ and $\phi_R \in C^{\infty}(\mathbb{R}^N, \mathbb{R})$ such that
$$
\phi_R=0, \quad x \in B_{\frac{R}{2}}(0), \quad \phi_R=1, \quad x \in B_R^c(0), \quad 0 \leq  \phi_R \leq  1 \quad \text { and } \quad\left|\nabla \phi_R\right| \leq \frac{C}{R},
$$
where $C>0$ is a constant independent of $R$. As the sequence $(\|\phi_R u_n\|_{\lambda_n, R_n})$ is bounded, it follows that
$$
\langle I_{\lambda_n, R_n}^{\prime}(u_n), \phi_R u_n \rangle=o_n(1),
$$
that is,
\begin{eqnarray*}
&&\int_{\mathbb{R}^N}\left(\left|\nabla_A u_n\right|^2+\left(\lambda_n V(x)+1\right)\left|u_n\right|^2\right) \phi_R d x\\
&= & \mathrm{Re}\left(\int_{\Omega_{\Gamma}^{\prime}} F_2^{\prime}\left(u_n\right)\phi_R \overline{ u_n} d x +\int_{\mathbb{R}^N \backslash \Omega_{\Gamma}^{\prime}} \widetilde{F}_2^{\prime}\left(u_n\right) \phi_R\overline{ u_n} d x-\int_{\mathbb{R}^N} F_1^{\prime}\left(u_n\right)\phi_R \overline{ u_n} d x\right) \\
&& -\mathrm{Re}\left(\int_{\mathbb{R}^N} \overline{u_n} \nabla u_n \nabla \phi_R d x\right)+o_n(1) .
\end{eqnarray*}
Let $R>0$ such that $\Omega_{\Gamma}^{\prime} \subset B_{\frac{R}{2}}(0)$, using H\"older's inequality and the boundedness of the sequence $(\|u_n\|_{\lambda_n, R_n})$ in $\mathbb{R}$, we get
\begin{eqnarray*}
\int_{\mathbb{R}^N}\left(\left|\nabla_A u_n\right|^2+\left(\lambda_n V(x)+1\right)\left|u_n\right|^2\right)\phi_R d x &\leq&  b_0 \int_{\mathbb{R}^N}\left|u_n\right|^2 \phi_Rd x +\frac{C}{R}\left\|u_n\right\|_{\lambda_n, R_n}^2+o_n(1)\\
&\leq&\frac{1}{2}\int_{\mathbb{R}^N}\left(\lambda_n V(x)+1\right)\left|u_n\right|^2 \phi_Rd x +\frac{C}{R}+o_n(1) .
\end{eqnarray*}
Hence, fixing $\xi>0$ and passing to the limit, we have
$$
\limsup _{n \rightarrow \infty} \int_{\mathbb{R}^N \backslash B_R(0)}\left(\left|\nabla_A u_n\right|^2+\left(\lambda_n V(x)+1\right)\left|u_n\right|^2\right) d x \leqslant \frac{C}{R}<\xi
$$
for some $R$ sufficiently large.
Note that $G_2^{\prime}$ has a subcritical growth, it follows that
$$
\begin{aligned}
 \mathrm{Re}\left(\int_{\mathbb{R}^N} G_2^{\prime}\left(x, u_n\right)\overline{ w} d x\right) &\rightarrow \mathrm{Re}\left(\int_{\mathbb{R}^N} G_2^{\prime}\left(x, u\right) \overline{w} d x\right), \quad \forall w \in C_0^{\infty}(\mathbb{R}^N), \\
 \int_{\mathbb{R}^N} G_2^{\prime}\left(x, u_n\right) \overline{u_n} d x &\rightarrow \int_{\mathbb{R}^N} G_2^{\prime}\left(x, u\right) \overline{u} d x\
\end{aligned}
$$
and
$$
\int_{\mathbb{R}^N} G_2\left(x, u_n\right) d x \rightarrow \int_{\mathbb{R}^N} G_2\left(x, u\right) d x.
$$
Now, recalling that $\lim\limits_{n \rightarrow \infty} \langle I_{\lambda_n, R_n}^{\prime}\left(u_n\right), w\rangle=0$ for all $w \in C_0^{\infty}(\mathbb{R}^N)$ and $\left\|u_n\right\|_{\lambda_n, R_n}^2 \leqslant K, \forall n \in \mathbb{N}$, we deduce that
$$
\mathrm{Re}\left(\int_{\mathbb{R}^N}(\nabla_A u \overline{\nabla_A \omega}+u \overline{\omega}) d x+\int_{\mathbb{R}^N} F_1^{\prime}(u) \overline{\omega} d x\right)=\mathrm{Re}\left(\int_{\mathbb{R}^N} G_2^{\prime}\left(x, u\right) \overline{\omega} d x\right),
$$
so
$$
\int_{\mathbb{R}^N}\left(|\nabla_A u|^2+|u|^2\right) d x+\int_{\mathbb{R}^N} F_1^{\prime}(u) \overline{u} d x=\int_{\mathbb{R}^N} G_2^{\prime}\left(x, u\right)\overline{ u} d x .
$$
Together this equality and $\lim\limits_{n \rightarrow \infty} \langle I_{\lambda_n, R_n}^{\prime}\left(u_n\right), u_n\rangle=0$, that is,
\begin{eqnarray*}
&&\int_{\mathbb{R}^N}\left(\left|\nabla_A u_n\right|^2+\left(\lambda_n V(x)+1\right)\left|u_n\right|^2\right) d x+\int_{\mathbb{R}^N} F_1^{\prime}\left(u_n\right) \overline{u_n} d x\\
&=&\int_{\mathbb{R}^N} G_2^{\prime}\left(x, u_n\right) \overline{u_n} d x+o_n(1)
\end{eqnarray*}
leads to
\begin{eqnarray*}
&& \lim\limits_{n \rightarrow+\infty}\left(\int_{\mathbb{R}^N}\left(\left|\nabla_A u_n\right|^2+\left(\lambda_n V(x)+1\right)\left|u_n\right|^2\right) d x+ \int_{\mathbb{R}^N} F_1^{\prime}\left(u_n\right) \overline{u_n} d x \right) \\
& =&\int_{\mathbb{R}^N}\left(|\nabla_A u|^2+|u|^2\right) d x+ \int_{\mathbb{R}^N} F_1^{\prime}(u) \overline{u} d x ,
\end{eqnarray*}
up to subsequence if necessary, we have
$$
u_n \rightarrow u \quad \text { in } H_A^1(\mathbb{R}^N), \quad \lambda_n \int_{\mathbb{R}^N} V(x)\left|u_n\right|^2 d x \rightarrow 0
$$
and
$$
  F_1^{\prime}\left(u_n\right) \overline{u_n}  \rightarrow   F_1^{\prime}(u) \overline{u}  \quad \text { in } L^1(\mathbb{R}^N) .
$$
Since $F_1$ is convex, even and $F(0)=0$, we know that $F_1^{\prime}(t) \overline{t} \geq  F_1(t) \geq  0$ for all $t \in \mathbb{C}$. Hence, using Lebesgue dominated convergence theorem, it holds
$$
F_1\left(u_n\right) \rightarrow F_1(u)\quad \text { in } L^1(\mathbb{R}^N).
$$
Due to
$$
\left\|u_n-u\right\|_{\lambda_n, R_n}^2=\int_{\mathbb{R}^N}\left|\nabla_A u_n-\nabla_A u\right|^2 d x+\int_{\mathbb{R}^N}\left|u_n-u\right|^2 d x+\lambda_n \int_{\mathbb{R}^N} V(x)\left|u_n-u\right|^2 d x,
$$
it follows that
$$
\left\|u_n-u\right\|_{\lambda_n, R_n}^2 \rightarrow 0,
$$
which implies (i) holds. The proofs of (ii) and (iii) are similar to that of Lemma \ref{L3.5}, so we omit it.
\end{proof}

Next, we will investigate the boundedness outside $\Omega_{\Gamma}^{\prime}$ for the solutions of \eqref{eq3.1}.
\begin{lemma}\label{L3.7}
Assume $\left(u_{\lambda, R}\right)$ be a family of nontrivial solutions of \eqref{eq3.1} satisfying $\left(I_{\lambda, R}\left(u_{\lambda, R}\right)\right)$ is bounded in $\mathbb{R}$ for any $\lambda>0$ and $R>0$ large enough. Then there exists $K>0$ that does not depend on $\lambda>0$ and $R>0$, and $R^*>0$ such that
$$
\left\|u_{\lambda, R}\right\|_{\infty, R} \leq  K, \quad \forall \lambda>0   , \  R \geqslant R^*.
$$
\end{lemma}
\begin{proof}
Consider $\lambda>0, L>0$ and $\beta>1$, let
$$
\begin{aligned}
& |u_{L, \lambda}|:= \begin{cases}|u_{\lambda, R}|, & \text { if } |u_{\lambda, R}| \leq  L, \\
L, & \text { if } |u_{\lambda, R}| \geq  L,\end{cases} \\
& |z_{L, \lambda}|=|u_{L, \lambda}|^{2(\beta-1)} |u_{\lambda, R}| \text { and } \omega_{L, \lambda}=|u_{\lambda, R}| |u_{L, \lambda}|^{\beta-1} .
\end{aligned}
$$
Since $\left(u_{\lambda, R}\right)$ is a nontrivial solution to \eqref{eq3.1}, it follows that
\begin{equation*}
  -(\nabla+i A(x))^2 u_{\lambda, R}+(\lambda V(x)+1) u_{\lambda, R}=g_2(x, u_{\lambda, R} )-F_1^{\prime}(u_{\lambda, R}).
\end{equation*}
By using Kato's inequality
$$
\Delta\left|u_{\lambda, R}\right| \geq \operatorname{Re}\left(\frac{\overline{u_{\lambda, R}}}{\left|u_{\lambda, R}\right|}(\nabla+i A(x))^2 u_{\lambda, R}\right),
$$
we obtain that
$$
-\Delta\left|u_{\lambda, R}\right|-(\lambda V(x)+1)\left|u_{\lambda, R}\right| \leq g_2(x, |u_{\lambda, R}|)-F_1^{\prime}(|u_{\lambda, R}|), \quad x \in \mathbb{R}^N.
$$
Now, choosing $z_{L, \lambda}$ as a test function, it holds
\begin{eqnarray}\label{eq3.10}
&& \int_{B_R(0)} |u_{L, \lambda}|^{2(\beta-1)}\left|\nabla |u_{\lambda, R}|\right|^2 d x+2(\beta-1) \int_{B_R(0)} |u_{L, \lambda}|^{2 \beta-3} |u_{\lambda, R}| \nabla |u_{\lambda, R}| \nabla |u_{L, \lambda}| d x \nonumber\\
&& +\int_{B_R(0)}(\lambda V(x)+1) |u_{L, \lambda}|^{2(\beta-1)}\left|u_{\lambda, R}\right|^2 d x+\int_{B_R(0)} F_1^{\prime}\left(|u_{\lambda, R}|\right)| u_{L, \lambda}|^{2(\beta-1)}| u_{\lambda, R}| d x \nonumber\\
&\leq&\int_{B_R(0)} G_2^{\prime}\left(x, |u_{\lambda, R}|\right) |u_{L, \lambda}|^{2(\beta-1)} |u_{\lambda, R}| d x .
\end{eqnarray}
According to the definition of $G_2$,
\begin{equation}\label{eq3.11}
  G_2^{\prime}(x, t) \leqslant F_2^{\prime}(t) \leqslant C t^{p-1}, \quad \forall(x, t) \in \mathbb{R}^N \times \mathbb{R}^{+},\ p \in\left(2,2^*\right).
\end{equation}
Hence, by using \eqref{eq3.10} and \eqref{eq3.11},
\begin{eqnarray}\label{eq3.12}
\int_{B_R(0)}\left(\left|\nabla \omega_{L, \lambda}\right|^2+\left|\omega_{L, \lambda}\right|^2\right) d x &\leq&  C \int_{B_R(0)} |u_{\lambda, R}|^p |u_{L, \lambda}|^{2(\beta-1)} d x\nonumber\\
&=&C \int_{B_R(0)} |u_{\lambda, R}|^{p-2} |\omega_{L, \lambda}|^2 d x .
\end{eqnarray}
By H\"older inequality, it is easy to see that
\begin{equation}\label{eq3.13}
\int_{B_R(0)} |u_{\lambda, R}|^{p-2} |\omega_{L, \lambda}|^2 d x \leq  C \beta^2\left(\int_{B_R(0)} |u_{\lambda, R}|^p d x\right)^{\frac{p-2}{p}}\left(\int_{B_R(0)} |\omega_{L, \lambda}|^p d x\right)^{\frac{2}{p}} .
\end{equation}
Moreover, it follows from Sobolev inequality that
\begin{equation}\label{eq3.14}
\left(\int_{B_R(0)}\left|\omega_{L, \lambda}\right|^{2^*} d x\right)^{\frac{2}{2^*}} \leq  C \int_{B_R(0)}\left(\left|\nabla \omega_{L, \lambda}\right|^2+\left|\omega_{L, \lambda}\right|^2\right) d x .
\end{equation}
Combining \eqref{eq3.12}-\eqref{eq3.14}, it holds
$$
\left(\int_{B_R(0)}\left|\omega_{L, \lambda}\right|^{2^*} d x\right)^{\frac{2}{2^*}} \leqslant C \beta^2\left(\int_{B_R(0)} |u_\lambda|^{p \beta} d x\right)^{\frac{2}{p}} .
$$
Regarding the use of the Fatou's lemma for variable $L$, we know that
$$
\left(\int_{B_R(0)}\left|u_\lambda\right|^{2^* \beta} d x\right)^{\frac{2}{2^*}} \leq  C \beta^2\left(\int_{B_R(0)} |u_\lambda|^{p \beta} d x\right)^{\frac{2}{p}},
$$
so
\begin{equation}\label{eq3.15}
  \left(\int_{B_R(0)}\left|u_\lambda\right|^{2^* \beta} d x\right)^{\frac{1}{2^* \beta}} \leq  C^{\frac{1}{ \beta}} \beta^{\frac{1}{\beta}}\left(\int_{B_R(0)} |u_\lambda|^{p \beta} d x\right)^{\frac{1}{2 \beta}  } .
\end{equation}
Since $\left(I_{\lambda, R}\left(u_{\lambda, R}\right)\right)$ is bounded and $\left(u_{\lambda, R}\right)$ is a solution of \eqref{eq3.1}, similar to the Lemma \ref{L3.2}, there exists $C>0$ satisfying
$$
\|u_{\lambda, R}\|_{\lambda, R} \leq  C
$$
for $\lambda>0$ and $R>0$ large enough. Fixing any sequences $\lambda_n \rightarrow+\infty$ and $R_n \rightarrow+\infty$, it is easy to see that $\left(u_{\lambda_n, R_n}\right)$ satisfies the hypotheses from Lemma \ref{L3.6}, and then $u_{\lambda_n, R_n} \rightarrow u$ in $H_A^1(\mathbb{R}^N)$. Now, note that $2<p<2^*$ and $(\|u_{\lambda_n, R_n}\|_{L^{2^*}(\mathbb{R}^N)})$ is bounded in $\mathbb{R}$, a well-known iteration argument (see \cite[Lemma 3.10]{{COACJ12020}} and \eqref{eq3.15} imply that there exists a constant $K_1>0$ such that
$$
\|u_{\lambda_n, R_n}\|_{L^{\infty}\left(\mathbb{R}^N\right)} \leqslant K_1, \quad \forall n \in \mathbb{N} .
$$
The proof is completed.
\end{proof}
\begin{lemma}\label{L3.8}
Assume $\left(u_{\lambda, R}\right)$ be a family of nontrivial solutions of \eqref{eq3.1} with $I_{\lambda, R}\left(u_{\lambda, R}\right)$ bounded in $\mathbb{R}$ for any $\lambda>0$ and $R>0$ large enough. Then there exist $\lambda^{\prime}>0$ and $R^{\prime}>0$ satisfying
$$
\left\|u_{\lambda, R}\right\|_{\infty, B_R(0) \backslash \Omega_{\Gamma}^{\prime}} \leqslant a_0, \quad \forall \lambda \geqslant \lambda^{\prime}, \quad R \geqslant R^{\prime} .
$$
Moreover, $u_{\lambda, R}$ is a solution of the original problem \eqref{eq3.2} for any $\lambda \geq  \lambda^{\prime}$ and $R \geq  R^{\prime}$.
\end{lemma}
\begin{proof}
Take $R_0>0$ large enough satisfying $\overline{\Omega_{\Gamma}^{\prime}} \subset B_{R_0}(0)$ and fix a neighborhood $\mathcal{B}$ of $\partial \Omega_{\Gamma}^{\prime}$ satisfying
$$
\mathcal{B} \subset B_{R_0}(0) \backslash \Omega_{\Gamma}
$$
By using Moser's iteration technique, there exists $C>0$ independent of $\lambda$ such that
$$
\|u_{\lambda, R}\|_{L^{\infty}(\partial \Omega_{\Gamma}^{\prime})} \leq  C\|u_{\lambda, R}\|_{L^{2^*}(\mathcal{B})}, \quad \forall R \geq  R_0 .
$$
Fixing two sequences $\lambda_n \rightarrow+\infty$ and $R_n \rightarrow+\infty$ and using Lemma \ref{L3.6}, it follows that $u_{\lambda_n, R_n} \rightarrow 0$ in $H_A^1(B_{R_n}(0) \backslash \Omega_{\Gamma})$ for some subsequence, and then $u_{\lambda_n, R_n} \rightarrow 0$ in $H_A^1(B_{R_0}(0) \backslash \Omega_{\Gamma})$, so
$$
\|u_{\lambda_n, R_n}\|_{L^{2^*}(\mathcal{B})} \rightarrow 0 \quad \text { as } n \rightarrow \infty .
$$
Hence, there exists an $n_0 \in \mathbb{N}$ such that
$$
\|u_{\lambda_n, R_n}\|_{L^{\infty}(\partial \Omega_{\Gamma}^{\prime})} \leqslant a_0, \quad \forall n \geqslant n_0 .
$$
Now, for $n \geqslant n_0$ we set $\widetilde{u}_{\lambda_n, R_n}: B_{R_n}(0) \backslash \Omega_{\Gamma}^{\prime} \rightarrow \mathbb{C}$ given by
$$
|\widetilde{u}_{\lambda_n, R_n}(x)|=\left(|u_{\lambda_n, R_n}(x)|-a_0\right)^{+} .
$$
Therefore, $\widetilde{u}_{\lambda_n, R_n}(x) \in H_A^1\left(B_{R_n}(0) \backslash \Omega_{\Gamma}^{\prime}\right)$. We will show that $\widetilde{u}_{\lambda_n, R_n}(x)=0$ in $B_{R_n}(0) \backslash \Omega_{\Gamma}^{\prime}$, because this can get that
$$
\|u_{\lambda_n, R_n}\|_{\infty, B_{R_n}(0) \backslash \Omega_{\Gamma}^{\prime}} \leq  a_0 .
$$
In fact, taking $\widetilde{u}_{\lambda_n, R_n}$ as a test function and extending $\widetilde{u}_{\lambda_n, R_n}(x)=0$ in $\Omega_{\Gamma}^{\prime}$, it holds
\begin{eqnarray*}
&&\mathrm{Re}\left(\int_{B_{R_n}(0) \backslash \Omega_{\Gamma}^{\prime}} \nabla_A u_{\lambda_n, R_n} \overline{\nabla_A \widetilde{u}_{\lambda_n, R_n}} d x+\int_{B_{R_n}(0) \backslash \Omega_{\Gamma}^{\prime}}\left(\lambda_n V(x)+1\right) u_{\lambda_n, R_n} \overline{\widetilde{u}_{\lambda_n, R_n}} d x\right)\\
&\leq& \mathrm{Re}\left( \int_{B_{R_n}(0) \backslash \Omega_{\Gamma}^{\prime}} \widetilde{F}_2^{\prime}\left(u_{\lambda_n, R_n}\right) \overline{\widetilde{u}_{\lambda_n, R_n}} d x\right) .
\end{eqnarray*}
Since
\begin{equation*}
  \int_{B_{R_n}(0) \backslash \Omega_{\Gamma}^{\prime}} \nabla_A u_{\lambda_n, R_n} \overline{\nabla_A \widetilde{u}_{\lambda_n, R_n} }d x=\int_{B_{R_n}(0) \backslash \Omega_{\Gamma}^{\prime}}\left|\nabla \widetilde{u}_{\lambda_n, R_n}\right|^2 d x,
\end{equation*}
 \begin{eqnarray*}
&&\mathrm{Re}\left(\int_{B_{R_n}(0) \backslash \Omega_{\Gamma}^{\prime}}\left(\lambda_n V(x)+1\right) u_{\lambda_n, R_n} \overline{\widetilde{u}_{\lambda_n, R_n}} d x\right)\\
&=&\mathrm{Re}\left(\int_{\left(B_{R_n}(0) \backslash \Omega_{\Gamma}^{\prime}\right)_{+}}\left(\lambda_n V(x)+1\right)\left(\widetilde{u}_{\lambda_n, R_n}+a_0\right) \overline{\widetilde{u}_{\lambda_n, R_n}} d x\right)
 \end{eqnarray*}
and
 \begin{eqnarray*}
&&\mathrm{Re}\left(\int_{B_{R_n}(0) \backslash \Omega_{\Gamma}^{\prime}} \widetilde{F}_2^{\prime}\left(u_{\lambda_n, R_n}\right) \overline{\widetilde{u}_{\lambda_n, R_n}} d x\right)\\
&=&\mathrm{Re}\left(\int_{\left(B_{R_n}(0) \backslash \Omega_{\Gamma}^{\prime}\right)_{+}} \frac{\widetilde{F}_2^{\prime}\left(u_{\lambda_n, R_n}\right)}{u_{\lambda_n, R_n}}\left(\widetilde{u}_{\lambda_n, R_n}+a_0\right) \overline{\widetilde{u}_{\lambda_n, R_n}} d x\right),
 \end{eqnarray*}
where
$$
\left(B_{R_n}(0) \backslash \Omega_{\Gamma}^{\prime}\right)_{+}=\left\{x \in B_{R_n}(0) \backslash \Omega_{\Gamma}^{\prime}: |u_{\lambda_n, R_n}(x)|>a_0\right\}.
$$
According to the above equalities, it holds
\begin{eqnarray*}
&&\mathrm{Re}\left(\int_{\left(B_{R_n}(0) \backslash \Omega_{\Gamma}^{\prime}\right)_{+}}\left(\left(\lambda_n V(x)+1\right)-\frac{\widetilde{F}_2^{\prime}\left(u_{\lambda_n, R_n}\right)}{u_{\lambda_n, R_n}}\right)\left(\widetilde{u}_{\lambda_n, R_n}+a_0\right)\overline{ \widetilde{u}_{\lambda_n, R_n} }d x\right)\\
&&+\int_{B_{R_n}(0) \backslash \Omega_{\Gamma}^{\prime}}\left|\nabla \widetilde{u}_{\lambda_n, R_n}\right|^2 d x\leq0 .
\end{eqnarray*}
According to the definition of $\widetilde{F}_2^{\prime}$, we know that
$$
\left(\lambda_n V(x)+1\right)-\frac{\widetilde{F}_2^{\prime}\left(u_{\lambda_n, R_n}\right)}{u_{\lambda_n, R_n}} \geq  1-b_0>0 \quad \text { in }\left(B_{R_n}(0) \backslash \Omega_{\Gamma}^{\prime}\right)_{+} .
$$
Thus, $\widetilde{u}_{\lambda_n, R_n}=0$ in $\left(B_{R_n}(0) \backslash \Omega_{\Gamma}^{\prime}\right)_{+}$and $B_{R_n}(0) \backslash \Omega_{\Gamma}^{\prime}$. Therfore, there exist $\lambda^{\prime}>0$ and $R^{\prime}>0$ satisfying
$$
\left\|u_{\lambda, R}\right\|_{\infty, B_R(0) \backslash \Omega_{\Gamma}^{\prime}} \leq  a_0, \quad \forall \lambda \geq \lambda^{\prime}, \quad R \geqslant R^{\prime}.
$$
The proof is completed.
\end{proof}

\section{Minimax level}
In this subsection, for any $\lambda>0$ and $j \in \Gamma$, let us denote by $\mathcal{I}_j: H_A^1(\Omega_j) \rightarrow \mathbb{R}$ and $\mathcal{I}_{\lambda, j}: H_A^1(\Omega_j^{\prime}) \rightarrow \mathbb{R}$ the functionals given by
$$
\begin{aligned}
& \mathcal{I}_j(u)=\frac{1}{2} \int_{\Omega_j}\left(|\nabla_A u|^2+|u|^2\right) d x+\frac{1}{q}\int_{\Omega_j}|u|^qdx-\frac{1}{2} \int_{\Omega_j} |u|^2 \log |u|^2 d x, \\
& \mathcal{I}_{\lambda, j}(u)=\frac{1}{2} \int_{\Omega_j^{\prime}}\left(|\nabla_A u|^2+(\lambda V(x)+1)|u|^2\right) d x+\frac{1}{q}\int_{\Omega_j}|u|^qdx-\frac{1}{2} \int_{\Omega_j^{\prime}} |u|^2 \log |u|^2 d x,
\end{aligned}
$$
which are energy functionals related to the following logarithmic equations
\begin{equation}\label{eq4.1}
  \begin{cases}-(\nabla+iA(x))^2u=|u|^{q-2}u+u \log |u|^2, & \text { in } \Omega_j \\ u=0, & \text { on } \partial \Omega_j\end{cases}
\end{equation}
and
\begin{equation}\label{eq4.2}
  \begin{cases}-(\nabla+iA(x))^2u+\lambda V(x) u=|u|^{q-2}u+u \log |u|^2, & \text { in } \Omega_j^{\prime} \\ \frac{\partial u}{\partial \eta}=0, & \text { on } \partial \Omega_j^{\prime}.\end{cases}
\end{equation}
It is obvious that $\mathcal{I}_j$ and $\mathcal{I}_{\lambda, j}$ satisfy the mountain pass geometry. Due to $\Omega_j$ and $\Omega_j^{\prime}$ are bounded, and $\mathcal{I}_j$ and $\mathcal{I}_{\lambda, j}$ satisfy the (PS) condition, Using the same arguments in Section 3, there exist two nontrivial functions $\omega_j \in H_A^1(\Omega_j)$ and $\omega_{\lambda, j} \in H_A^1(\Omega_j^{\prime})$ satisfying
$$
\mathcal{I}_j\left(\omega_j\right)=c_j, \quad \mathcal{I}_{\lambda, j}\left(\omega_{\lambda, j}\right)=c_{\lambda, j} \quad \text { and } \quad \mathcal{I}_j^{\prime}\left(\omega_j\right)=\mathcal{I}_{\lambda, j}^{\prime}\left(\omega_{\lambda, j}\right)=0
$$
where
$$
\begin{aligned}
& c_j=\inf _{\gamma \in \Upsilon_j} \max _{t \in[0,1]} I_j(\gamma(t)), \\
& c_{\lambda, j}=\inf _{\gamma \in \Upsilon_{\lambda, j}} \max _{t \in[0,1]} I_{\lambda, j}(\gamma(t))
\end{aligned}
$$
and
$$
\begin{aligned}
& \Upsilon_j=\left\{\gamma \in C\left([0,1], H_0^1\left(\Omega_j\right)\right): \gamma(0)=0 \text { and } I_j(\gamma(1))<0\right\}, \\
& \Upsilon_{\lambda, j}=\left\{\gamma \in C\left([0,1], H^1\left(\Omega_j^{\prime}\right)\right): \gamma(0)=0 \text { and } I_{\lambda, j}(\gamma(1))<0\right\}
\end{aligned}
$$
In fact, a simple computation gives
$$
\begin{aligned}
& c_j=\inf _{u \in \mathcal{N}_j} \mathcal{I}_j(u), \\
& c_{\lambda, j}=\inf _{u \in \mathcal{N}_j^{\prime}} \mathcal{I}_{\lambda, j}(u),
\end{aligned}
$$
where
$$
\mathcal{N}_j=\left\{u \in H_A^1\left(\Omega_j\right) \backslash\{0\}: \mathcal{I}_j^{\prime}(u) u=0\right\}
$$
and
$$
\mathcal{N}_j^{\prime}=\left\{u \in H_A^1\left(\Omega_j^{\prime}\right) \backslash\{0\}: \mathcal{I}_{\lambda, j}^{\prime}(u) u=0\right\} .
$$
In addition, by using a direct computation, there exists a $\tau>0$ satisfying if $u \in \mathcal{N}_j$  for any $j \in \Gamma$, then
\begin{equation}\label{eq4.3}
  \|u\|_j>\tau,
\end{equation}
where $\|\cdot\|_j$ is given as follows
$$
\|u\|_j^2=\int_{\Omega_j}\left(|\nabla_A u|^2+|u|^2\right) d x.
$$
Particularly, since $\omega_j \in \mathcal{N}_j$, we must have $\left\|\omega_{\lambda, j}\right\|_j>\tau$, where $\omega_{\lambda, j}=\left.\omega_j\right|_{\Omega_j}$ for all $j \in \Gamma$.

In the following, $c_{\Gamma}=\sum\limits_{j=1}^l c_j$ and $T>0$ is a constant large enough, which does not depend on $\lambda$ and $R>0$ large enough, such that
\begin{equation}\label{eq4.4}
  0<\langle \mathcal{I}_j^{\prime}(\frac{1}{T} \omega_j), \frac{1}{T} \omega_j\rangle, \quad \langle\mathcal{I}_j^{\prime}\left(T \omega_j\right), T \omega_j \rangle<0, \quad \forall j \in \Gamma .
\end{equation}
Therefore, according to the definition of $c_j$, we have
$$
\max _{s \in\left[1 / T^2, 1\right]} I_j\left(s T \omega_j\right)=c_j, \quad \forall j \in \Gamma .
$$
Without loss of generality, consider $\Gamma=\{1,2, \ldots, l\}$ with $l \leq  k$ and fix
$$
\begin{aligned}
& \gamma_0\left(s_1, s_2, \ldots, s_l\right)(x)=\sum_{j=1}^l s_j T \omega_j(x), \quad \forall\left(s_1, s_2, \ldots, s_l\right) \in\left[1 / T^2, 1\right]^l, \\
& \Gamma_*=\left\{\gamma \in C\left(\left[1 / T^2, 1\right]^l, E_{\lambda, R} \backslash\{0\}\right): \gamma=\gamma_0 \text { on } \partial\left(\left[1 / T^2, 1\right]^l\right)\right\}.
\end{aligned}
$$
and
$$
b_{\lambda, R, \Gamma}=\inf _{\gamma \in \Gamma_*} \max _{\left(s_1, s_2, \ldots, s_l\right) \in\left[1 / T^2, 1\right]^l} I_{\lambda, R}\left(\gamma\left(s_1, s_2, \ldots, s_l\right)\right) .
$$
Note that $\gamma_0 \in \Gamma_*$, then $\Gamma_* \neq \emptyset$ and $b_{\lambda, R, \Gamma}$ is well defined.

\begin{lemma}\label{L4.1}
For each $\gamma \in \Gamma_*$, there exists $\left(t_1, t_2, \ldots, t_l\right) \in\left[1 / T^2, 1\right]^l$ satisfying
$$
\mathcal{I}_{\lambda, j}^{\prime}\left(\gamma\left(t_1, \ldots, t_l\right)\right) \gamma\left(t_1, \ldots, t_l\right)=0 \quad \text { for } j \in\{1, \ldots, l\}.
$$
\end{lemma}
\begin{proof}
Let $\gamma \in \Gamma_*$, we study the map $\widetilde{\gamma}:\left[1 / T^2, 1\right]^l \rightarrow \mathbb{R}^l$ defined as follows
$$
\widetilde{\gamma}\left(s_1, \ldots, s_l\right)=\left(I_{\lambda, 1}^{\prime}\left(\gamma\left(s_1, \ldots, s_l\right)\right) \gamma\left(s_1, \ldots, s_l\right), \ldots, I_{\lambda, l}^{\prime}\left(\gamma\left(s_1, \ldots, s_l\right)\right) \gamma\left(s_1, \ldots, s_l\right)\right) .
$$
For any $\left(s_1, \ldots, s_l\right) \in \partial\left(\left[1 / T^2, 1\right]^l\right)$, it holds
$$
\gamma\left(s_1, \ldots, s_l\right)=\gamma_0\left(s_1, \ldots, s_l\right)
$$
Hence, by using \eqref{eq4.4} and Miranda's theorem (see \cite{MC1940}), this completes the proof.
\end{proof}
\begin{lemma}\label{L4.2}

(a) For any $\lambda>0$ and $R>0$ large enough, $\sum\limits_{j=1}^l c_{\lambda, j} \leqslant b_{\lambda, R, \Gamma} \leqslant c_{\Gamma}$.

(b) For $\gamma \in \Gamma_*$ and $\left(s_1, \ldots, s_l\right) \in \partial\left(\left[1 / T^2, 1\right]^l\right)$,  then
$$
I_{\lambda, R}\left(\gamma\left(s_1, \ldots, s_l\right)\right)<c_{\Gamma}, \quad \forall \lambda>0
$$
\end{lemma}
\begin{proof}
The proof of the lemma is the same as that of \cite[Proposition 4.2]{{COA2006}}.
\end{proof}
\begin{lemma}\label{L4.3}
(a) $b_{\lambda, R, \Gamma}$ is a critical value of $I_{\lambda, R}$ for $\lambda>0$ and $R>0$ large enough.

(b) $b_{\lambda, R, \Gamma} \rightarrow c_{\Gamma}$, when $\lambda \rightarrow+\infty$ uniformly for $R>0$ large enough.
\end{lemma}
\begin{proof}
The proof of the corollary is similar to that of \cite[Corollary 4.3]{{COA2006}}.
\end{proof}
\section{Uniform estimate}
Now, define $I_{\lambda, R}^{c_{\Gamma}}$ and $\Theta$ as follows:
$$
I_{\lambda, R}^{c_{\Gamma}}:=\left\{u \in E_{\lambda, R}: I_{\lambda, R}(u) \leqslant c_{\Gamma}\right\}
$$
and
$$
\Theta:=\left\{u \in E_{\lambda, R}:\|u\|_{\lambda, \Omega_j^{\prime}}>\frac{\tau}{2 T}, \forall j \in \Gamma\right\},
$$
where $\tau$ and $T$ were fixed in \eqref{eq4.3} and \eqref{eq4.4}, respectively. Fixing $\kappa=\frac{\tau}{8 T}$ and $\mu>0$, we define
$$
A_{\mu, R}^\lambda=\left\{u \in \Theta_{2 \kappa}: I_{\lambda, B_R(0) \backslash \Omega_{\Gamma}^{\prime}}(u) \geqslant 0,\|u\|_{\lambda, B_R(0) \backslash \Omega_{\Gamma}}^2 \leq \mu,\left|\mathcal{I}_{\lambda, j}(u)-c_j\right| \leqslant \mu, \forall j \in \Gamma\right\},
$$
where $\Theta_r$ for $r>0$ denotes the set
$$
\Theta_r=\left\{u \in E_{\lambda, R}: \inf _{v \in \Theta}\|u-v\|_{\lambda, \Omega_j^{\prime}} \leqslant r, \forall j \in \Gamma\right\} .
$$
Notice that $w=\sum\limits_{j=1}^l w_j \in A_{\mu, R}^\lambda \cap I_{\lambda, R}^{c_{\Gamma}}$, which shows that $A_{\mu, R}^\lambda \cap I_{\lambda, R}^{c_{\Gamma}} \neq \emptyset$.

Now, we will establish uniform estimate of $\left\|I_{\lambda, R}^{\prime}(u)\right\|$ in the set $\left(A_{2 \mu, R}^\lambda \backslash A_{\mu, R}^\lambda\right) \cap I_{\lambda, R}^{c_{\Gamma}}$.

\begin{lemma}\label{L5.1}
For each $\mu>0$, there exist $\Lambda_*>0, R^*>0$ large enough and $\sigma_0>0$ independent of $\lambda$ and $R>0$ large enough such that
$$
\left\|I_{\lambda, R}^{\prime}(u)\right\| \geq \sigma_0 \quad \text { for } \lambda \geq  \Lambda_*, \quad R \geq  R^* \quad \text { and } \quad u \in\left(A_{2 \mu, R}^\lambda \backslash A_{\mu, R}^\lambda\right) \cap I_{\lambda, R}^{c_{\Gamma}} .
$$
\end{lemma}
\begin{proof}
Arguing by contradiction, we assume that there exist $\lambda_n, R_n \rightarrow \infty$ and $u_n \in\left(A_{2 \mu, R_n}^{\lambda_n} \backslash A_{\mu, R_n}^{\lambda_n}\right)$ $\cap I_{\lambda_n, R_n}^{c_{\Gamma}}$ such that
$$
\left\|I_{\lambda_n, R_n}^{\prime}\left(u_n\right)\right\| \rightarrow 0 .
$$
Since $u_n \in A_{2 \mu, R_n}^{\lambda_n}$, we know that $(\left\|u_n\right\|_{\lambda_n, R_n})$ and $\left(I_{\lambda_n, R_n}\left(u_n\right)\right)$ are both bounded. Then up to a subsequence if necessary, assume that $\left(I_{\lambda_n, R_n}\left(u_n\right)\right)$ is a convergent sequence. Hence, by Lemma \ref{L3.6}, there exists $u \in H_A^1\left(\Omega_{\Gamma}\right)$ such that $u$ is a solution for \eqref{eq4.1} and
$$
u_n \rightarrow u \quad \text { in } H_A^1(\mathbb{R}^N), \quad\left\|u_n\right\|_{\lambda_n, B_{R_n}(0) \backslash \Omega_{\Gamma}}^2 \rightarrow 0 \quad \text { and } \quad I_{\lambda_n, R_n}\left(u_n\right) \rightarrow \mathcal{I}_{\Gamma}(u) \in\left(-\infty, c_{\Gamma}\right] .
$$
Note that $\left(u_n\right) \subset \Theta_{2 \kappa}$, it holds
$$
\left\|u_n\right\|_{\lambda_n, \Omega_j^{\prime}}^2>\frac{\tau}{4 T}, \quad \forall j \in \Gamma.
$$
Letting $n \rightarrow+\infty$, we obtain that
$$
\|u\|_j^2 \geqslant \frac{\tau}{4 T}>0, \quad \forall j \in \Gamma,
$$
which implies $\left.u\right|_{\Omega_j} \neq 0, j=1, \ldots, l$ and $I_{\Gamma}^{\prime}(u)=0$. As a result, by using \eqref{eq4.3},
$$
\|u\|_j^2>\frac{\tau}{2 T}>0, \quad \forall j \in \Gamma .
$$
In this way, $\mathcal{I}_{\Gamma}(u) \geqslant c_{\Gamma}$. But according to the fact that $I_{\lambda_n, R_n}\left(u_n\right) \leq c_{\Gamma}$ and $I_{\lambda_n, R_n}\left(u_n\right) \rightarrow \mathcal{I}_{\Gamma}(u)$ as $n \rightarrow+\infty$, it follows that $\mathcal{I}_{\Gamma}(u)=c_{\Gamma}$. Therefore, for $n$ sufficiently large,
$$
\left\|u_n\right\|_j^2>\frac{\tau}{2 T}, \quad\left|I_{\lambda_n, R_n}\left(u_n\right)-c_{\Gamma}\right| \leq  \mu \quad \text { for any } j \in \Gamma \text {. }
$$
Thereby, $u_n \in A_{\mu, R_n}^{\lambda_n}$ for large $n$, which is a contradiction to $u_n \in\left(A_{2 \mu, R_n}^{\lambda_n} \backslash A_{\mu, R_n}^{\lambda_n}\right)$. This complete the proof.
\end{proof}
Now, define $\mu_1$ and $\mu^*$ as follows:
$$
\min _{t \in \partial\left[1 / T^2, 1\right]^l}\left|I_{\Gamma}\left(\gamma_0(t)\right)-c_{\Gamma}\right|=\mu_1>0
$$
and
$$
\mu^*=\min \left\{\mu_1, \kappa, r / 2\right\},
$$
where $\kappa=\frac{\tau}{8 T}$ was given before and $r>\max \left\{\left\|w_j\right\|_{H_0^1\left(\Omega_j\right)}: j=1, \ldots, l\right\}$. For each $s>0$, we also define
$$
B_s^\lambda:=\left\{u \in E_\lambda\left(B_R(0)\right):\|u\|_{\lambda, R} \leqslant s\right\} \text { for } s>0 .
$$
\begin{lemma}\label{L5.2}
Assume $\mu \in\left(0, \mu^*\right), \Lambda_*>0$ and $R^*>0$ sufficiently large as given in Lemma \ref{L5.1}. Then there exists a nontrivial solution $u_{\lambda, R}$ of \eqref{eq3.1} such that $u_\lambda \in A_{\mu, R}^\lambda \cap I_{\lambda, R}^{c_{\Gamma}} \cap B_{r+1}^\lambda$ for $\lambda \geqslant \Lambda_*$ and $R \geq  R^*$.
\end{lemma}
\begin{proof}
Assume that there exist no critical points for the functional $I_{\lambda, R}(u)$ in $A_{\mu, R}^\lambda \cap I_{\lambda, R}^{c_{\Gamma}} \cap B_{r+1}^\lambda$ for $\lambda \geq  \Lambda_*$, our goal is to look for a contradiction. Since $I_{\lambda, R}$ satisfies the (PS) condition, there exists a constant $d_\lambda>0$ satisfying
$$
\left\|I_{\lambda, R}^{\prime}(u)\right\| \geq  d_\lambda \quad \text { for all } u \in A_{\mu, R}^\lambda \cap I_{\lambda, R}^{c_{\Gamma}} \cap B_{r+1}^\lambda.
$$
By using Lemma \ref{L5.1},
$$
\left\|I_{\lambda, R}^{\prime}(u)\right\| \geq  \sigma_0 \quad \text { for all } u \in\left(A_{2 \mu, R}^\lambda \backslash A_{\mu, R}^\lambda\right) \cap I_{\lambda, R}^{c_{\Gamma}},
$$
where $\sigma_0>0$ is independent of $\lambda$. Now, define $\Psi: E_{\lambda, R} \rightarrow \mathbb{R}$ is a continuous functional such that
$$
\begin{aligned}
& \Psi(u)=1 \quad \text { for } u \in A_{3 \mu / 2, R}^\lambda \cap \Upsilon_\kappa \cap B_r^\lambda, \\
& \Psi(u)=0 \quad \text { for } u \notin A_{2 \mu, R}^\lambda \cap \Upsilon_{2 \kappa} \cap B_{r+1}^\lambda, \\
& 0 \leqslant \Psi(u) \leqslant 1, \quad \forall u \in E_{\lambda, R},
\end{aligned}
$$
and $H: I_{\lambda, R}^{c_{\Gamma}} \rightarrow E_\lambda\left(B_R(0)\right)$ is a function given by
$$
H(u):= \begin{cases}-\Psi(u)\frac{Y(u)}{\|Y(u)\|}, &u\in A_{2\mu,R}^{\lambda}\cap B_{r+1}^{\lambda}, \\ 0, & u\notin A_{2\mu,R}^{\lambda}\cap B_{r+1}^{\lambda},\end{cases}
$$
where $Y$ is a pseudo-gradient vector field for $I_{\lambda, R}$ on $\mathcal{K}=\left\{u \in E_{\lambda, R}:I_{\lambda, R}^{\prime}(u) \neq 0\right\}$. It is obvious that $H$ is well defined since $I_{\lambda, R}^{\prime}(u) \neq 0$ for $u \in A_{2 \mu, R}^\lambda \cap \Phi_{\lambda, R}^{c_{\Gamma}}$. Note that
$$
\|H(u)\| \leqslant 1, \quad \forall \lambda \geqslant \Lambda_* \quad \text { and } \quad u \in \Phi_{\lambda, R}^{c_{\Gamma}},
$$
so
\begin{equation}\label{eq5.1}
 \frac{d}{d t} I_{\lambda, R}(\eta(t, u)) \leq -\Psi(\eta(t, u))\left\|I_{\lambda, R}^{\prime}(\eta(t, u))\right\| \leq  0,
\end{equation}
\begin{equation*}
  \left\|\frac{d \eta}{d t}\right\|_\lambda=\|H(\eta)\|_\lambda \leq  1,
\end{equation*}
\begin{equation}\label{eq5.2}
  \eta(t, u)=u \quad \text { for all } t \geqslant 0 \quad \text { and } \quad u \in I_{\lambda, R}^{c_{\Gamma}} \backslash\left(A_{2 \mu, R}^\lambda \cap B_{r+1}^\lambda\right),
\end{equation}
where the deformation flow $\eta:[0, \infty) \times I_{\lambda, R}^{c_{\Gamma}} \rightarrow I_{\lambda, R}^{c_{\Gamma}}$ defined by
$$
\frac{d \eta}{d t}=H(\eta) \quad \text { and } \quad \eta(0, u)=u \in I_{\lambda, R}^{c_{\Gamma}}.
$$
Next, we consider two paths:

(1) The path $t \rightarrow \eta\left(t, \gamma_0(t)\right)$, where $t=\left(t_1, \ldots, t_l\right) \in\left[1 / T^2, 1\right]^l$.

If $\mu \in\left(0, \mu^*\right)$, we have
$$
\gamma_0(t) \notin A_{2 \mu, R}^\lambda, \quad \forall t \in \partial\left(\left[1 / T^2, 1\right]^l\right) .
$$
As $I_{\lambda, R}\left(\gamma_0(t)\right) \leq  c_{\Gamma}$ for any $t \in \partial\left(\left[1 / T^2, 1\right]^l\right)$, by using \eqref{eq5.2}, we have
$$
\eta\left(t, \gamma_0(t)\right)=\gamma_0(t), \quad \forall t \in \partial\left(\left[1 / T^2, 1\right]^l\right)
$$
Hence, $\eta\left(t, \gamma_0(t)\right) \in \Gamma_*$ for all $t \geq  0$.

(2) The path $t \rightarrow \gamma_0(t)$, where $t=\left(t_1, \ldots, t_l\right) \in\left[1 / T^2, 1\right]^l$.
In view of $\operatorname{supp} (t)  \subset \overline{\Omega_{\Gamma}}$ for all $t \in\left[1 / T^2, 1\right]^l, I_{\lambda, R}\left(\gamma_0(t)\right)$ does not depend on $\lambda>0$. Note that,
$$
I_{\lambda, R}\left(\gamma_0(t)\right) \leq  c_{\Gamma}, \quad \forall t \in\left[1 / T^2, 1\right]^l
$$
and
$$
I_{\lambda, R}\left(\gamma_0(t)\right)=c_{\Gamma} \Leftrightarrow t_j=1 / T, \quad \forall j \in \Gamma,
$$
so
$$
m_0:=\sup \left\{I_{\lambda, R}(u): u \in \gamma_0\left(\left[1 / T^2, 1\right]^l\right) \backslash A_\mu^\lambda\right\}
$$
is independent of $\lambda, R>0$ and $m_0<c_{\Gamma}$. Now, observing that there exists a $K_*>0$ satisfying
$$
\left|I_{\lambda, R}(u)-I_{\lambda, R}(v)\right| \leq  K_*\|u-v\|_{\lambda, R}, \quad \forall u, v \in B_r^\lambda.
$$
Next, we will prove that if $T_*>0$ is large enough, the estimate below holds:
\begin{equation}\label{eq5.3}
  \max _{t \in\left[1 / T^2, 1\right]^l} I_\lambda\left(\eta\left(T_*, \gamma_0(t)\right)\right)<\max \left\{m_0, c_{\Gamma}-\frac{1}{2 K_*} \sigma_0 \mu\right\}.
\end{equation}
Indeed, write $u=\gamma_0(t), t \in\left[1 / T^2, 1\right]^l$. If $u \notin A_{\mu, R}^\lambda$, using \eqref{eq5.2}, we must have that
$$
I_{\lambda, R}(\eta(t, u)) \leq  I_\lambda(\eta(0, u))=I_{\lambda, R}(u) \leqslant m_0, \quad \forall t \geq 0 .
$$
If $u \in A_{\mu, R}^\lambda$, let $\widetilde{\eta}(t)=\eta(t, u), \widetilde{d}_\lambda:=\min \left\{d_\lambda, \sigma_0\right\}$ and $T_*=\frac{\sigma_0 \mu}{2 K_* \widetilde{d}_\lambda}>0$, now we distinguish two cases:

(1) $\widetilde{\eta}(t) \in A_{3 \mu / 2, R}^\lambda \cap \Theta_\kappa \cap B_r^\lambda, \forall t \in\left[0, T_*\right]$.

(2) $\widetilde{\eta}\left(t_0\right) \notin A_{3 \mu / 2, R}^\lambda \cap \Theta_\kappa \cap B_r^\lambda$ for some $t_0 \in\left[0, T_*\right]$.

Assume (1) holds, it follows that $\Psi(\widetilde{\eta}(t)) \equiv 1$ and $\left\|I_{\lambda, R}^{\prime}(\widetilde{\eta}(t))\right\| \geq  \widetilde{d}_\lambda$ for all $t \in\left[0, T_*\right]$. Using \eqref{eq5.1}, we have
 \begin{eqnarray*}
I_{\lambda, R}\left(\widetilde{\eta}\left(T_*\right)\right)&=&I_{\lambda, R}(u)+\int_0^{T_*} \frac{d}{d s} I_{\lambda, R}(\widetilde{\eta}(s)) d s\\
& \leq& c_{\Gamma}-\int_0^{T_*} \widetilde{d}_\lambda d s \\
& =&c_{\Gamma}-\widetilde{d}_\lambda T_* \\
& \leq& c_{\Gamma}-\frac{\sigma_0 \mu}{2 K_*} .
 \end{eqnarray*}

If (2) holds, There are three situations.

(i) There exists $t_2 \in\left[0, T_*\right]$ satisfying $\widetilde{\eta}\left(t_2\right) \notin \Theta_\kappa$, and thus for $t_1=0$, it holds
$$
\left\|\widetilde{\eta}\left(t_2\right)-\widetilde{\eta}\left(t_1\right)\right\|_{\lambda, R} \geqslant \delta>\mu
$$
since $\widetilde{\eta}\left(t_1\right)=u \in \Theta$.

(ii) There exists $t_2 \in\left[0, T_*\right]$ such that $\widetilde{\eta}\left(t_2\right) \notin B_r^\lambda$, so that for $t_1=0$, we obtain
$$
\left\|\widetilde{\eta}\left(t_2\right)-\widetilde{\eta}\left(t_1\right)\right\|_{\lambda, R} \geq r>\mu,
$$
because $\widetilde{\eta}\left(t_1\right)=u \in B_r^\lambda$.

(iii) $\widetilde{\eta}(t) \notin \Theta_\kappa \cap B_r^\lambda$, and there exist $t_1$ and $t_2$ satisfying $0 \leq  t_1<t_2 \leq  T_*$ such that $\widetilde{\eta}(t) \in A_{3 \mu / 2, R}^\lambda \backslash A_{\mu, R}^\lambda$ for all $t \in\left[t_1, t_2\right]$ with
$$
\left|I_{\lambda, R}\left(\widetilde{\eta}\left(t_1\right)\right)-c_{\Gamma}\right|=\mu \quad \text { and } \quad\left|I_{\lambda, R}\left(\widetilde{\eta}\left(t_2\right)\right)-c_{\Gamma}\right|=\frac{3 \mu}{2}.
$$
According to the definition of $K_*$,
$$
\begin{aligned}
\left\|\widetilde{\eta}\left(t_2\right)-\widetilde{\eta}(t)\right\|_{\lambda, R} & \geq  \frac{1}{K_*}\left|I_{\lambda, R}\left(\widetilde{\eta}\left(t_2\right)\right)-I_{\lambda, R}\left(\widetilde{\eta}\left(t_1\right)\right)\right| \\
& \geq  \frac{1}{K_*}\left(\left|I_{\lambda, R}\left(\widetilde{\eta}\left(t_2\right)\right)-c_{j_0}\right|-\left|I_{\lambda, R}\left(\widetilde{\eta}\left(t_1\right)\right)-c_{j_0}\right|\right) \\
& \geq  \frac{1}{2 K_*} \mu .
\end{aligned}
$$
Using the mean value theorem and $t_2-t_1 \geq  \frac{1}{2 K_*} \mu$, it follows that
$$
\begin{aligned}
I_{\lambda, R}\left(\widetilde{\eta}\left(T_*\right)\right) & =I_{\lambda, R}(u)+\int_0^{T_*} \frac{d}{d s} I_{\lambda, R}(\widetilde{\eta}(s)) d s \\
& \leq  I_{\lambda, R}(u)-\int_0^{T_*} \Psi(\widetilde{\eta}(s))\left\|I_{\lambda, R}^{\prime}(\widetilde{\eta}(s))\right\| d s \\
& \leq  c_{\Gamma}-\int_{t_1}^{t_2} \sigma_0 d s \\
& =c_{\Gamma}-\sigma_0\left(t_2-t_1\right) \\
& \leq c_{\Gamma}-\frac{\sigma_0 \mu}{2 K_*}
\end{aligned}
$$
which implies \eqref{eq5.3} holds.

Fixing $\widehat{\eta}(t)=\eta\left(T_*, \gamma_0(t)\right)$, it holds $\widehat{\eta}(t) \in \Theta_{2 \kappa}$, so $\left.\widehat{\eta}(t)\right|_{\Omega_j^{\prime}} \neq 0$ for all $j \in \Gamma$. Hence, $\widehat{\eta} \in \Gamma_*$ and
$$
b_{\lambda, R, \Gamma} \leq  \max _{s \in\left[1 / T^2, 1\right]^l} I_{\lambda, R}(\widehat{\eta}(s)) \leq \max \left\{m_0, c_{\Gamma}-\frac{\sigma_0 \mu}{2 K_*}\right\}<c_{\Gamma}.
$$
But by using Lemma \ref{L4.3}, $b_{\lambda, R, \Gamma} \rightarrow c_{\Gamma}$ as $\lambda \rightarrow \infty$ uniformly holds for $R>0$ large enough, this is obviously contradictory.
Therefore, the conclusion of lemma holds.
\end{proof}
\section{The proof of Theorem \ref{t1.1}}
Using Lemma \ref{L5.2}, for $\mu \in\left(0, \mu^*\right)$ and $\Lambda_*>0$, we can find a nontrivial solution $u_{\lambda, R}$ for the problem \eqref{eq3.1} such that $u_{\lambda, R} \in A_{\mu, R}^\lambda \cap I_{\lambda, R}^{c_{\Gamma}} \cap B_{r+1}^\lambda$ for all $\lambda \geq  \Lambda_*$ and $R \geq R^*$.

Now fix $\lambda \geq  \Lambda_*$ and let $R_n \rightarrow+\infty$, there exists a solution $u_{\lambda, n}=u_{\lambda, R_n}$ for \eqref{eq3.1} with
$$
u_{\lambda, n} \in A_{\mu, R_n}^\lambda \cap I_{\lambda, R_n}^{c_{\Gamma}} \cap B_{r+1}^\lambda, \quad \forall n \in \mathbb{N} .
$$
Since $(u_{\lambda, n})$ is bounded in $H_A^1(\mathbb{R}^N)$, we can assume that for some $u_\lambda \in H_A^1(\mathbb{R}^N)$,
$$
\begin{aligned}
& I_{\lambda, R_n}(u_{\lambda, n}) \rightarrow d \leq  c_{\Gamma}, \\
& u_{\lambda, n} \rightarrow u_\lambda \quad \text { in } H_A^1(\mathbb{R}^N), \\
& u_{\lambda, n} \rightarrow u_\lambda \quad \text { in } L_{\text {loc }}^q(\mathbb{R}^N) \quad \text { for any } q \in\left[1,2^*\right)
\end{aligned}
$$
and
$$
u_{\lambda, n}(x) \rightarrow u_\lambda(x) \quad \text { a.e. } x \in \mathbb{R}^N .
$$
Recalling Lemma \ref{L3.8}, we have
$$
0 \leq  |u_{\lambda, n}(x)| \leq  a_0, \quad \forall x \in \mathbb{R}^N \backslash \Omega_{\Gamma},
$$
so
$$
0 \leq  |u_\lambda(x)| \leq  a_0, \quad \forall x \in \mathbb{R}^N \backslash \Omega_{\Gamma} .
$$
Next, we will introduce two lemmas. Their proofs follow from the similar arguments in the proof of Lemma \ref{L3.6}, so we omit them.
\begin{lemma}\label{L6.1}
For any fixed $\xi>0$, there exists an $R>0$ satisfying
$$
\limsup\limits_{n \rightarrow \infty} \int_{\mathbb{R}^N \backslash B_R(0)}\left(\left|\nabla_A u_{\lambda, n}\right|^2+(\lambda V(x)+1)\left|u_{\lambda, n}\right|^2\right) d x \leq  \xi .
$$
\end{lemma}
\begin{lemma}\label{L6.2}
$u_{\lambda, n} \rightarrow u_\lambda$ in $H_A^1(\mathbb{R}^N)$. In addition,
$$
F_1\left(u_{\lambda, n}\right) \rightarrow F_1\left(u_\lambda\right) \  \text {and} \ F_1^{\prime}\left(u_{\lambda, n}\right) u_{\lambda, n} \rightarrow F_1^{\prime}\left(u_\lambda\right) u_\lambda \  \text {in}\ L^1(\mathbb{R}^N) .
$$
\end{lemma}
Therefore, consider the following energy functional $I_\lambda: E_\lambda \rightarrow(-\infty,+\infty]$,
$$
I_\lambda(u)=\frac{1}{2} \int_{\mathbb{R}^N}\left(|\nabla_A u|^2+(\lambda V(x)+1)|u|^2\right) d x-\frac{1}{2} \int_{\mathbb{R}^N} |u|^2 \log |u|^2 d x
$$
it is easy to see that $u_\lambda$ is a critical point of $I_\lambda$ satisfying
$$
u_\lambda \in A_\mu^\lambda=\left\{u \in\left(\Theta_{\infty}\right)_{2 \kappa}: I_{\lambda, \mathbb{R}^N \backslash \Omega_{\Gamma}^{\prime}}(u) \geq  0,\|u\|_{\mathbb{R}^N \backslash \Omega_{\Gamma}}^2 \leq  \mu,\left|I_{\lambda, j}(u)-c_j\right| \leq  \mu, \forall j \in \Gamma\right\}
$$
where
$$
\Theta_{\infty}=\left\{u \in E_\lambda:\|u\|_{\lambda, \Omega_j^{\prime}}>\frac{\tau}{2 T}, \forall j \in \Gamma\right\}
$$
and
$$
\left(\Theta_{\infty}\right)_r=\left\{u \in E_\lambda: \inf _{v \in \Theta_{\infty}}\|u-v\|_{\lambda, \Omega_j^{\prime}} \leq  r, \forall j \in \Gamma\right\} .
$$

\begin{proof}[\bf Proof of Theorem \ref{t1.1} ] Let $\lambda_n \rightarrow+\infty$ and $\mu_n \in\left(0, \mu^*\right)$ with $\mu_n \rightarrow 0$, we can find a solution $u_n \in A_{\mu_n}^{\lambda_n}$ of the problem \eqref{eq1.1} with $\lambda=\lambda_n$. Hence, $\left(u_n\right)$ is bounded in $H_A^1(\mathbb{R}^N)$ such that

(a) $\left\|I_{\lambda_n}^{\prime}\left(u_{\lambda_n}\right)\right\|=0, \forall n \in \mathbb{N}$;

(b) $\left\|u_{\lambda_n}\right\|_{\lambda_n, \mathbb{R}^N \backslash \Omega_{\Gamma}} \rightarrow 0$;

(c) $I_{\lambda_n}\left(u_n\right) \rightarrow d \leqslant c_{\Gamma}$,\\
where
$$
\left\|I_\lambda^{\prime}(u)\right\|=\sup \left\{\left\langle I_\lambda^{\prime}(u), z\right\rangle: z \in H_A^1\left(\mathbb{R}^N\right) \text { and }\|z\|_\lambda \leq  1\right\}.
$$
According to the proof in Lemma \ref{L3.6}, there exists a $u \in H_A^1(\mathbb{R}^N)$ satisfying $u_{\lambda_n} \rightarrow u$ in $H_A^1(\mathbb{R}^N)$, and $u \equiv 0$ in $\mathbb{R}^N \backslash \Omega_{\Gamma}$ and $u$ is a nontrivial solution of
\begin{equation}\label{eq6.1}
\begin{cases}-(\nabla+iA(x))^2u=|u|^{q-2}u+u \log |u|^2 & \text { in } \Omega_{\Gamma}, \\  u=0 & \text { on } \partial \Omega_{\Gamma},\end{cases}
\end{equation}
so $I_{\Gamma}(u) \geq  c_{\Gamma}$. Moreover, note that $I_{\lambda_n}\left(u_{\lambda_n}\right) \rightarrow \mathcal{I}_{\Gamma}(u)$, then we get $\mathcal{I}_{\Gamma}(u)=d$ and $d \geq  c_{\Gamma}$. due to $d \leq  c_{\Gamma}$, it follows that $\mathcal{I}_{\Gamma}(u)=c_{\Gamma}$, which implies that $u$ is a least energy solution for \eqref{eq6.1}. This completes the proof of the theorem.
\end{proof}

\end{document}